\documentclass[reqno, natbib]{amsart}

\allowdisplaybreaks
\usepackage{natbib,latexsym,amsmath,amssymb,color,enumerate,graphicx,url}
\usepackage[noprefix]{nomencl}
\usepackage{epstopdf}

\numberwithin{equation}{section}

\usepackage[]{hyperref}
            
\hypersetup{colorlinks,
            linkcolor=blue,
            anchorcolor=blue,
            citecolor=blue}

\newtheorem{thm}{Theorem}[section] 
 
\newtheorem{myprop}[thm]{Proposition}
\newtheorem{mylem}[thm]{Lemma}
\newtheorem{mycor}[thm]{Corollary}

\def\fidi{\xrightarrow{fidi}}
\def\fidichf{fidi chf }
\def\a{\alpha}
\def\aD{\a_{D}}
\def\aR{\a_{R}}
\def\mD{\mu_{D}}
\def\mR{\mu_{R}}
\def\FD{F_{D}}
\def\FR{F_{R}}
\def\tFD{\bar{F}_{D}}
\def\tFR{\bar{F}_{R}}

\def\l{\lambda}

\def\e{\epsilon}
\def\d{\delta}
\def\z{\zeta}

\def\R{\mathbb{R}}

\def\PRM{\mathbb{PRM}}

\def\Slow{\mathcal{S}}
\def\Moderate{\mathcal{M}}
\def\Fast{\mathcal{F}}

\newcommand{\cprm}[1]{{\buildrel _{\circ} \over #1}}

\begin{document}

\title[Heterogeneous traffic at large time scales]{On the superposition of heterogeneous traffic\\ at large time scales}

\author[L. L\'opez-Oliveros]{Luis L\'opez-Oliveros}
\address{Luis L\'opez-Oliveros\\
Department of Statistical Science\\
Cornell University \\
Ithaca, NY 14853}
\email{ll278@cornell.edu}

\author[S.I.\ Resnick]{Sidney I.\ Resnick}
\address{Prof. Sidney Resnick\\
School of Operations Research and Information Engineering\\
Cornell University \\
Ithaca, NY 14853}
\email{sir1@cornell.edu}

\begin{abstract}
{Various empirical and theoretical studies
indicate that} cumulative network traffic 
is a Gaussian process.
{ However, depending on whether the {intensity} at which
sessions are initiated is large or small relative to the session 
duration tail,}
\citet{mikosch:resnick:rootzen:stegeman:2002} and
\citet{kaj:taqqu:2008} have shown that traffic at large time scales
can be approximated by either fractional Brownian motion (fBm) or
stable L\'evy motion. 
 We study {distributional
properties of cumulative traffic that consists of a finite number
of independent streams and  } 
give an explanation of  why
Gaussian examples abound in practice but not stable L\'evy motion.
We offer an explanation of how much vertical
aggregation is needed for the Gaussian approximation to hold. Our
results are expressed as limit theorems for a sequence of cumulative
traffic processes whose {session initiation {intensities} satisfy growth
  rates similar to those used } in
\citet{mikosch:resnick:rootzen:stegeman:2002}. 
\end{abstract}

\maketitle

\section{Introduction}\label{sec:intro}

Collection of data network measurements often {uses}
 an algorithm for clustering packets with {the} same source and
 destination IP addresses. Various criteria for grouping packets yield
 different entities, e.g. connections, flows (or unidirectional
 connections), end-to-end streams, etc. \citep[See e.g.][Section
 4]{sarvotham:riedi:baraniuk:2005}. {These} high-order
 constructs {of packet clusters}  are sometimes termed
{sessions}. For now, think of a 
 session as a user downloading a file, streaming media, or accessing
 websites. For each session, summary measurements are computed for the 
{{\it size\/} (the number of bytes transmitted in a session)},
the duration
 of the session and the
 average transfer rate. Data sets {of these summaries} show some distinctive properties,
 such as heavy tails for session {size} and duration
 \citep{arlittson:williamson:1996,crovella:bestavros:1997,willinger:paxson:taqqu:1998a}
 and {sometimes} rate \citep{maulik:resnick:rootzen:2002,resnick:2003}. 

 Typically, {a time resolution or granularity is selected or
   imposed.} Typical resolutions are 1, 10 or 100 milliseconds, 1
 second, 1 minute, 1 hour, etc. {Once a resolution is fixed,}
 the number of bytes {or number of packets} per unit time {can be recorded
and cumulative network loads over stationary time intervals computed.}
{These cumulative loads have been
studied from empirical and theoretical perspectives with the objectives}
of  satisfying performance criterion and offering
 adequate bandwidth provisioning \citep{meent:mandjes:2005} or
 predicting properties of congestion events
 \citep{jin:bali:duncan:frost:2007}. 

{Conventional wisdom based on empirical studies claims}
 that a heavily loaded network link subject to 
 aggregation over many users  should see  Gaussian traffic. {This wisdom is
 considered a network {\it invariant\/}.}
Influential examples based on the Bellcore measurements
 \citep{leland:taqqu:willinger:wilson:1994} suggest that
 \emph{horizontal} aggregation, that is, working with a single on/off stream 
at sufficiently
 large time scale {justifies} Gaussian modeling. See also
 \citet{kurtz:1996} and \citet{willinger:taqqu:sherman:wilson:1997}.

However, mathematically {it is known that}
 with heavy tailed session durations, cumulative
{load} at large time scales can be approximated by either fractional
Brownian motion (fBm) or stable L\'evy motion, depending on whether
the {intensity} at which sessions are initiated is large or small relative to
the size of the duration tails. See \cite{
mikosch:resnick:rootzen:stegeman:2002,
kaj:taqqu:2008,
taqqu:willinger:sherman:1997}. {The stable approximation has not been
observed empirically}
\citep{guerin:nyberg:perrin:resnick:rootzen:starica:2003}
and  use of Gaussian cumulative loads has become {dominant}
\citep{kilpi:norros:2002,sarvotham:wang:riedi:baraniuk:2002,jain:dovrolis:2005}. 

{But why should traffic be Gaussian?}  {According to the empirical
study \citet{meent:mandjes:pras:2006}, in addition to horizontal
aggregation, the superposition of independent traffic streams, that
is, \emph{vertical} aggregation, can justify a Gaussian model and, in
fact, the number of traffic streams need not be large to make
cumulative loads approximately Gaussian.}

In this paper we
\begin{itemize}
\item study the distribution of the cumulative load in the presence of
  a finite number of independent traffic streams;
\item give an explanation for  why Gaussian examples abound in
  practice but not stable ones;
\item answer how much vertical aggregation is needed to justify the
  use of fBm. 
\end{itemize}
 
Our findings suggest that cumulative load for  aggregate traffic can be
approximated by fBm at large time
  scales {provided} the initiation {intensity} of at least one of the traffic
  components is large. 
Network traffic in the wild has several distinct constituents and 
we claim that in practice there is one or more
 components with dominant large initiation {intensities}. For example, this
 should be the case with  web traffic using port 80
and this suggests why  Gaussian traffic should be
pervasive \citep{meent:mandjes:pras:2006}. 

{Before discussing mathematical details,} we
{illustrate the phenomena of interest} with a {motivating}
 example of a network trace captured
at Cornell University main campus servers during 55 days between
November 2, 2009, and January 15, 2010. Cornell's data set is a
collection of \emph{netflow} records, where all non-IP traffic has
been discarded and only TCP and UPD traffic is present in the trace. A
netflow is a collection of packets with the same source and
destination IP addresses, source and destination ports, protocol,
ingress interface and IP type of service \citep{cisco:2007}. In our
data, TCP traffic accounts for nearly 90\% of the bytes, and over 80\%
of the total number of netflows, mostly port 80 (http traffic)
netflows. We have taken the part of the trace corresponding to both
outgoing and incoming traffic between 1 and 5 p.m. local time, adding
up to 220 hours of traffic. {The anonymization procedure used on the
data obliterated the distinction between outgoing and incoming
flows.}

We analyze the distribution of $A^{(TCP)}$ and $A^{(UDP)}$, namely the cumulative load generated by TCP and UDP bytes, respectively. For this purpose, we separate the trace into TCP and UDP netflows and for $k=1,\ldots,220$ we count
\begin{align*}
A_{k}^{(TCP)}&:=\textrm{total number of TCP bytes captured in the $k$th hour},\\
A_{k}^{(UDP)}&:=\textrm{total number of UDP bytes captured in the $k$th hour}.
\end{align*}
Due to the dates and times of collection, these counts exhibit both a trend and a daily seasonality. Here we detrend and remove daily seasonality \citep[see e.g.][Section 1.4]{brockwell:davis:1991}, but our conclusions are the same without this massage.

Figure \ref{fig:qqnormtcpudp} shows Gaussian QQ plots for $A^{(TCP)}$ (\textit{left}) and $A^{(UDP)}$ (\textit{right}). A straight line fit is evident for the TCP cumulative input. However, the UDP counterpart shows a significant departure from the straight line. Using the $p$-values of the Anderson-Darling two-sided test also shows no evidence against the normality of $A^{(TCP)}$ ($p=0.1369$), but strong evidence against a Gaussian model for $A^{(UDP)}$ ($p=9.8\times10^{-16}$).

\begin{figure}[htb]
\centering
\includegraphics{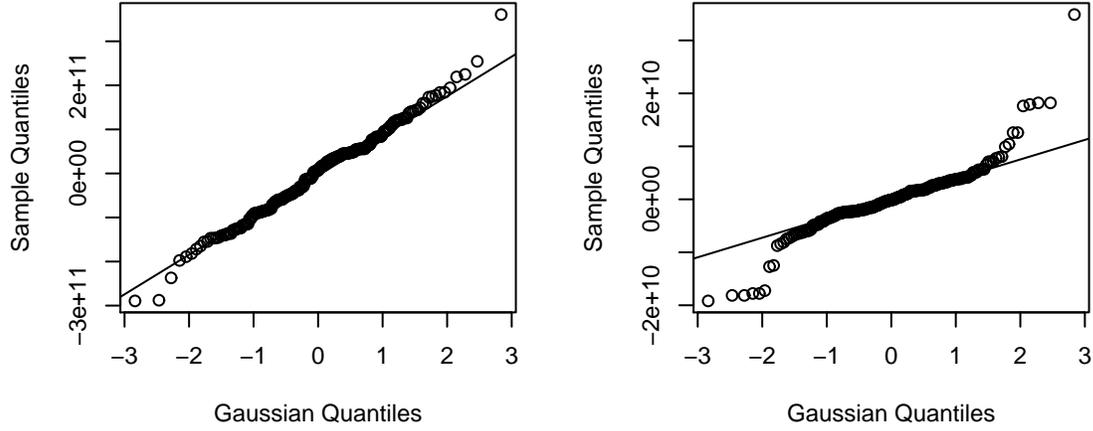}
\caption{Normal QQ plots of cumulative inputs. \textit{Left}: TCP traffic. \textit{Right}: UDP traffic}\label{fig:qqnormtcpudp} 
\end{figure}

We also check whether $A^{(UDP)}$ is a heavy-tailed random variable,
in the sense of its distribution tail being regularly varying with
tail index $\a$ \citep{dehaan:ferreira:2006, resnickbook:2007}. For
instance, Figure \ref{fig:plotudp} \textit{left} shows a stable regime
in the Hill plot of $\a$ \citep[for Hill plots, see
e.g.][]{hill:1975,dehaan:resnick:1998,
  resnickbook:2007}. Additionally, in Figure \ref{fig:plotudp}
\textit{right} we present the exponential QQ plot of $\log(A^{(UDP)})$
with a straight line fit through the biggest 55 observations. This
shows no evidence against approximating the distribution 
of  {thresholded} values 
of $A^{(UDP)}$ by a Pareto  (Recall that the
logarithm of Pareto random variable is exponential; see, for example,
\citet{mcneil:frey:embrechts:2005,resnickbook:2007, coles:2001}.) 

\begin{figure}[htb]
\centering
\includegraphics{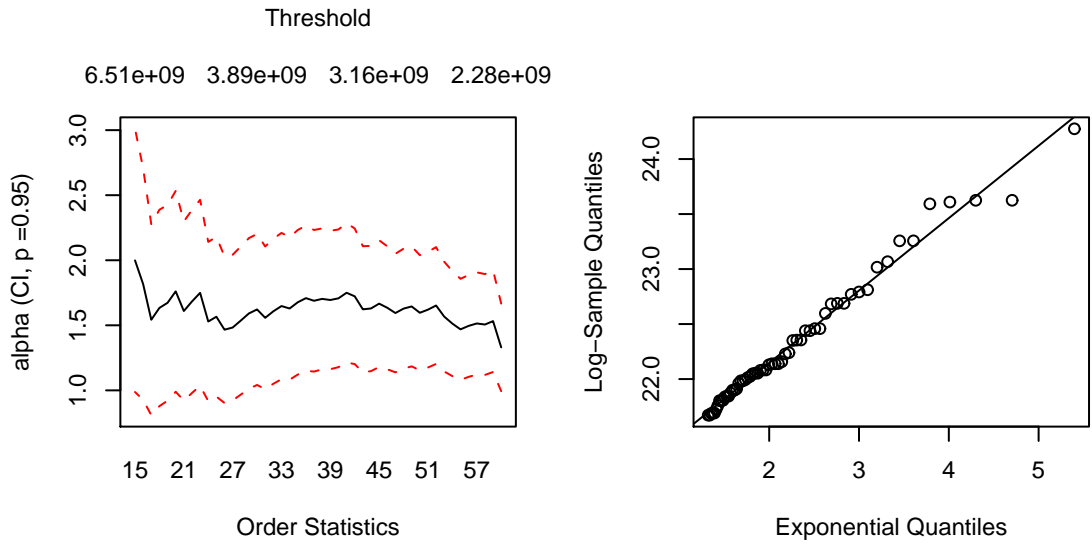}
\caption{Plots for the UDP cumulative input. \textit{Left:} Hill plot of tail index with 95\% confidence interval. \textit{Right:} Exponential QQ plot of the log data.}\label{fig:plotudp}
\end{figure}

{If} we consider the aggregated cumulative load,
 $A^{(TCP)}+A^{(UDP)}$, the normal QQ plot in Figure
\ref{fig:qqnorma} exhibits a straight line fit {and the}
Anderson-Darling test $p-$value is $0.2117$, showing no evidence to
reject normality. Without accounting for centering and scaling, this
result is rather counterintuitive due to the nature of the individual
tails of $A^{(TCP)}$ and $A^{(UDP)}$. 

\begin{figure}[htb]
\centering
\includegraphics{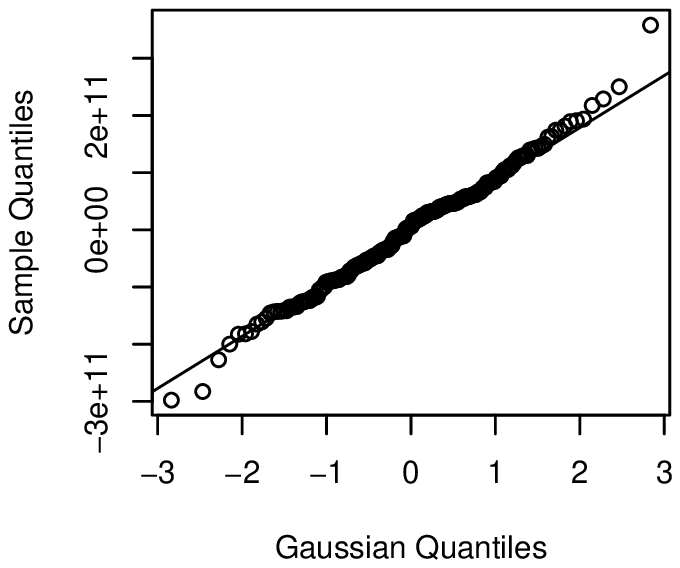}
\caption{Normal QQ plot of the aggregated cumulative input.}\label{fig:qqnorma}
\end{figure}

Our explanation to the above phenomenon starts by modeling the
quantity of data in windows of length $T$ in Section
\ref{sec:setup}. Analogously to the slow and fast growths of
\cite{mikosch:resnick:rootzen:stegeman:2002}, we define two different
scenarios for the aggregated traffic.
{A} third scenario is defined similarly to the boundary case considered
in \citet{kaj:taqqu:2008}. In Section \ref{sec:main} we obtain
approximations and provide clarification of the asymptotic
behavior at large time scales. {We let} $T\to\infty$ and see what
limits exist for the aggregated cumulative {load}. In Section
\ref{sec:extensions} we study  extensions to our model and
finally Section \ref{sec:techproofs} contains some technical results
used to prove our main theorems. 

\subsection{Notation}\label{sec:not}

In order to simplify the later presentation, we introduce and collect
some notation. References are provided for further reading.

\begin{tabular}{l l}
\noalign{\smallskip}\\
$\bar{F}$ & The right tail of the distribution function $F$, i.e. $\bar{F}=1-F$.\\\noalign{\smallskip}
$F^{\leftarrow}$ & The left continuous inverse of the distribution
function $F$,
\\\noalign{\smallskip} &
 i.e. $F^{\leftarrow}(y)=\inf\{x:F(x)\geq y\}$.\\\noalign{\smallskip}
$f_{1}\sim f_{2}$ & $\lim_{x\to\infty} f_{1}(x)/f_{2}(x) = 1$.\\\noalign{\smallskip}
$\fidi$ & Convergence of finite dimensional distributions.\\\noalign{\smallskip}
$\xrightarrow{v}$ & Vague convergence of measures. See
e.g. \citet{kallenberg:1984,
resnick:1987}.\\\noalign{\smallskip}
$M_{+}(0,\infty]$ & The space of nonnegative Radon measures on $(0,\infty]$.\\\noalign{\smallskip}
$RV_{\gamma}$ & The class of regularly varying functions with index $\gamma$.\\\noalign{\smallskip}
 & See e.g. \citet{bingham:goldie:teugels:1987,dehaan:ferreira:2006,resnick:1986}.\\\noalign{\smallskip}
$\xi=\PRM(E\xi)$ & A Poisson random measure $\xi$ with mean measure $E\xi$.\\\noalign{\smallskip}
$\cprm{\xi}$ & A compensated Poisson random measure with mean measure $E\xi$, i.e. $\cprm{\xi}=\xi-E\xi$.\\\noalign{\smallskip}
$N^{\infty}_{\gamma,h_{\d}}$ &
$N^{\infty}_{\gamma,h_{\d}}=\PRM(ds\cdot\gamma u^{-(\gamma+1)}du\cdot
h_{\d}(dr))$ on $\R\times(0,\infty)^{2}$, where $h_{\d}$ is a measure
on $(0,\infty)$
\\\noalign{\smallskip} 
 & such that $h_{\d}[\cdot,\infty)\in RV_{-\d}$. If $h_{\d}(dr)=\d
 r^{-(\d+1)}dr$, we may simply write
 $N^{\infty}_{\gamma,\d}$.\\\noalign{\smallskip} 
$N^{\infty}_{\gamma,\d}$ & $N^{\infty}_{\gamma,\d}=\PRM(ds\cdot \gamma u^{-(\gamma+1)}du\cdot\d r^{-(\d+1)}dr)$ on $\R\times(0,\infty)^{2}$.\\\noalign{\smallskip} 
$M_{\gamma,m}(dv)$ & A $\gamma-$stable random measure with control measure $m(dv)$ and stable index $1<\gamma<2$.\\\noalign{\smallskip}
 & For $\xi=\PRM(m(dv)w^{-(\rho+1)}dw)$, we can write\\\noalign{\smallskip}
  & $M_{\gamma,m}(A)\stackrel{d}{=}\left(\left(-\cos\frac{\pi\gamma}{2}\right)\frac{2\Gamma(2-\gamma)}{\gamma(\gamma-1)}\right)^{-1/\gamma}\int_{A}\int_{w=0}^{\infty}w\cprm{\xi}(dv,dw)$.\\\noalign{\smallskip}
 & See e.g. \citet[][Chapter 3]{samorodnitsky:taqqu:1994}.\\\noalign{\smallskip}
$\Lambda_{\gamma}(\cdot)$ & A $\gamma-$stable L\'evy motion totally skewed to the right with stable index $1<\gamma<2$. \\\noalign{\smallskip}
 & In general, we can write $\Lambda_{\gamma}(t)\stackrel{d}{=}\left(\left(-\cos\frac{\pi\gamma}{2}\right)\frac{2\Gamma(2-\gamma)}{\gamma(\gamma-1)}\right)^{1/\gamma}\int_{0}^{\infty}1_{\{0<v<t\}}M_{\rho,m}(dv)$. \\\noalign{\smallskip}
  & See e.g. \citet[][Chapter 3]{samorodnitsky:taqqu:1994}.\\\noalign{\smallskip}
$B_{H}(\cdot)$ & The standard fractional Brownian motion with Hurst exponent $H$.
\end{tabular}

\section{Model Description and Basic Assumptions}\label{sec:setup}

Consider a network that has an infinite number of nodes. At certain
times, a node begins a transmission session at a random rate {that is}
fixed throughout the session. Suppose network traffic consists of $p$
distinct types which we call \textit{streams}. In practice, such a
division of network traffic arises naturally; e.g. traffic can be
segmented by application type (web, email, streaming media,
file-sharing applications, etc.), by protocol (TCP, UDP, IMTP, etc.),
and even by users.  {We suppose the $p$ streams are independent
and that each follows an
$M/G/\infty$ input model. The overall load is obtained by  aggregating
over the $p$ streams.}
 Thus, the basic assumptions
are as follows: 

\begin{itemize}
\item Sessions corresponding to the $j$th stream are initiated at homogenous Poisson time points \linebreak $\{\Gamma_{k}^{(j)},-\infty<k<\infty\}$ with arrival intensity $\l^{(j)}>0$. These points are labeled so that $\Gamma_{0}^{(j)}<0<\Gamma_{1}^{(j)}$ whence $\{-\Gamma_{0}^{(j)},\Gamma_{1}^{(j)},(\Gamma_{k+1}^{(j)}-\Gamma_{k}^{(j)},k\not=0)\}$ are iid exponential with parameter $\l^{(j)}$. Thus, we have:
\begin{equation*}
\sum_{k}\e_{\Gamma_{k}^{(j)}}=\PRM(\l^{(j)}ds).
\end{equation*}
We assume that these $\PRM$s are independent.
\item All the sessions in the network transmit data at positive random
  rates that are iid {with common distribution} $\FR$. Let
  $\{R_{k}^{(j)}\}$ be the rate of the $k$th session of the $j$th
  stream. {Assume that} either $\tFR\in RV_{-\aR}, 1<\aR<2$, or $E[(
  R_{1}^{(1)} )^{2}]<\infty$. In either case, define
  $\mR:=ER_{1}^{(1)}$. 
\item {S}essions {in} the $j$th stream have positive durations
  $\{D_{k}^{(j)}\}$, $j=1,\ldots,p,$ that are iid $\FD^{(j)}$, with
  $\tFD^{(j)}\in RV_{-\aD^{(j)}}$, $1<\aD^{(j)}<2$, and
  $\mD^{(j)}:={ED^{(j)}_{1} }$. In general, not all the $\aD^{(j)}$s are
  equal. 
\item We also assume mutually independent durations across streams, and that durations and rates are independent.
\end{itemize}
There is empirical evidence justifying the choices of $\aD^{(j)}$s and
$\aR$: See
e.g. \citet{cunha:bestavros:crovella:1995,willinger:taqqu:leland:wilson:1995,leland:taqqu:willinger:wilson:1994,resnick:2003,lopez-oliveros:resnick:2009}. For
now, we adopt a network-centric approach by assuming the rate of
communication entirely depends on the state and speed of the
network. {Studies supporting } this assumption include
\citet{shakkottai:brownlee:claffy:2005} and
\citet{kortebi:muscariello:oueslati:roberts:2005}. 

We will need
\begin{equation}\label{eq:agl}
\l=\sum_{j=1}^{p}\l^{(j)},
\end{equation}
\begin{equation}\label{eq:fmix}
\FD:=\sum_{j=1}^{p} (\l^{(j)}/\l)\FD^{(j)},
\end{equation}
and the quantile functions
\begin{align}
b_{D}^{(j)}(t)&=(1/\tFD^{(j)})^{\leftarrow}(t)=(\FD^{(j)})^{\leftarrow}(1-1/t),\label{eq:bj}\\
b_{D}(t)&=(1/\tFD)^{\leftarrow}(t)=\FD^{\leftarrow}(1-1/t),\label{eq:bmix}\\
b_{R}(t)&=(1/\tFR)^{\leftarrow}(t)=\FR^{\leftarrow}(1-1/t).\label{eq:br}
\end{align}
Notice that $\FD$ is the mixture model of the durations of the $p$
streams, with weights $\l^{(j)}/\l$, $j=1,\ldots,p$. In fact, $\FD$
{is} the distribution of the duration of the sessions of the
aggregated stream, and $\l^{(j)}/\l$ {is} the proportion of the traffic
that consists of sessions from the $j$th stream. We {return to}
this interpretation later. 

Now consider {$(s,u,r)$ as a  generic Poisson point}
representing a session that starts at time $s$, {has} duration $u$ and
rate $r$. By augmentation, the counting function of the session {descriptors}
$(\Gamma_{k}^{(j)},D_{k}^{(j)},R_{k}^{(j)})$ of the $j$th stream on
$\R\times[0,\infty)^{2}$ is 
\begin{equation}\label{eq:countj}
N^{(j)}:=\sum_{k}\e_{(\Gamma_{k}^{(j)},D_{k}^{(j)},R_{k}^{(j)})}=\PRM(\l^{(j)}ds \FD^{(j)}(du)\FR(dr)),\quad j=1,\ldots,p.
\end{equation}
By independence, the counting function of the session descriptors
 of the aggregated stream is
\begin{align}\label{eq:agcount}
N:=\sum_{j=1}^{p}{N^{(j)}}&=\PRM\left(\l ds \sum_{j=1}^{p} (\l^{(j)}/\l)\FD^{(j)}(du)\FR(dr) \right)\notag\\
&=\PRM(\l ds \FD(du)\FR(dr) ).
\end{align}
Thus, the mean measures of the $N^{(j)}$ and $N$ are given by
\begin{align*}
EN^{(j)}(ds,du,dr)&:=\l^{(j)}ds F_{D}^{(j)}(du)F_{R}(dr),\quad j=1,\ldots,p,\\
EN(ds,du,dr)&:=\l dsF_{D}(du)F_{R}(dr).
\end{align*}

In addition, let
\begin{equation}\label{eq:length}
L_{t}(s,u)=\left|[0,t]\cap[s,s+u]\right|=\int_{0}^{t}1_{[s,s+u]}(y)dy=\int_{0}^{u}1_{[0,t]}(y+s)dy,
\end{equation}
be the length of the subinterval of $[0,t]$ during which the session
$(s,u,r)$ transmits data. {In Lemma
\ref{lem:lengthprop}, we  summarize several required properties of $L_{t}(s,u)$.}

For each $j$, define
\begin{align}
A^{(j)}(t)&:=\textrm{cumulative input {in $[0,t]$ from} the $j$th stream}\notag\\
&=\int_{-\infty}^{\infty}\int_{0}^{\infty}\int_{0}^{\infty}rL_{t}(s,u)N^{(j)}(ds,du,dr),\label{eq:inpj}\\
\intertext{and similarly}
A(t)&:=\textrm{cumulative input {in $[0,t]$} from the aggregated stream}\notag\\
&=\int_{-\infty}^{\infty}\int_{0}^{\infty}\int_{0}^{\infty}rL_{t}(s,u)N(ds,du,dr).\label{eq:aginp}
\end{align}
{(\citet{kaj:taqqu:2008} showed that these integrals are well defined 
using Campbell's theorem \citep[Section 3.2]{kingman:1993}.)}
{Also,}
\begin{equation*}
EA^{(j)}(t)=\l^{(j)}\mD^{(j)}\mR t, \quad
EA(t)=\sum_{j=1}^{p} \l^{(j)}\mD^{(j)}\mR t=\l\mD \mR t,
\end{equation*}
where
\begin{equation}\label{eq:mdmix}
\mD:=\sum_{j=1}^{p} (\l^{(j)}/\l)\mD^{(j)}
\end{equation}
is the mean of the mixture model of the durations of the different streams.

Observe that we can write the cumulative inputs {as}
linear drift plus {compensated} random Poisson fluctuation as follows: 
\begin{align}
A^{(j)}(t)&:=\l^{(j)}\mD^{(j)}\mR t+\int_{-\infty}^{\infty}\int_{0}^{\infty}\int_{0}^{\infty}rL_{t}(s,u)\cprm{N}^{(j)}(ds,du,dr),\label{eq:cinpj}\\
A(t)&:=\l\mD \mR t+\int_{-\infty}^{\infty}\int_{0}^{\infty}\int_{0}^{\infty}rL_{t}(s,u)\cprm{N}(ds,du,dr).\label{eq:caginp}
\end{align}

{After scaling time by $T$,}  we think of $A_{T}^{(j)}{:}=(A^{(j)}(Tt),t>0),
j=1,\ldots,p$ and $A_{T}{:}=(A(Tt),t>0)$ for large $T$, as the cumulative
inputs on large time scales. Thus, we consider a family of models
indexed by the  time scale parameter $T$ {and from now on }
 we let the
arrival intensities depend on $T$ {so that}
 $\l^{(j)}:=\l^{(j)}(T)$. If necessary, we let
$\l_{j}(T)\to\infty$ as $T\to\infty$ (see
\eqref{eq:lgrowth}). {Dependence of the arrival intensities on $T$ means}
 $\l$, $\FD$, $b_{D}^{(j)}$ and $b_{D}$ as defined in
\eqref{eq:agl}-\eqref{eq:bmix} depend on $T$ as well; however, notice
that the tail indices of the distribution of the duration, namely
$\aD^{(j)}$, remain independent of $T$. In practice, the fact that we
focus on the stream at a particular time period, say $[0,Tt]$, does
not affect the tail index of the distribution of the sessions
duration, which is in accordance with our assumptions. For
convenience, we {often} suppress the sub{script} $T$.

{Fix $j$, $1\leq j\leq p$} and {in the $T$th model,}
let $A_{cs}^{(j)}(t)$ be the centered and scaled cumulative input of
the $j$th stream in $[0,Tt]$, that is 
\begin{equation}\label{eq:csinpj}
A_{cs}^{(j)}(t):=\frac{A^{(j)}(Tt)-\l^{(j)}\mD^{(j)}\mR Tt}{a^{(j)}(T)},
\end{equation}
for a suitable $a_{j}(T)$ to be made precise below.
{Assuming $
\lim_{T\to\infty}\l^{(j)}T\tFD^{(j)}(T)$ exists,}
the asymptotic behavior of $A_{cs}^{(j)}(t)$ as $T\to\infty$, depends
on whether the arrival rate  is large, moderate, or small,
relative to the tail  of the duration.

\begin{thm}\label{thm:mikosch}\citep{mikosch:resnick:rootzen:stegeman:2002,kaj:taqqu:2008}.

For any {$1\leq j\leq p$}, 
consider the following three growth regimes of the arrival rate:
\begin{equation}\label{eq:lgrowth}
\lim_{T\to\infty}\l^{(j)}T\tFD^{(j)}(T)=\begin{cases}
\infty,&\textrm{fast-growth}.\\
c_{j}^{\aD^{(j)}-1},&\textrm{moderate-growth},\\
0,&\textrm{slow-growth},
\end{cases}
\end{equation}
where $c_{j}\in(0,\infty)$. (The form of the moderate-growth limit 
{facilitates}  a simple expression of the corresponding limit process.)
{A}ssume that either $\tFR\in RV_{-\aR}, \aR>\aD^{(j)}$ or $E[(
R_{1}^{(1)} )^{2}]<\infty$. (If $\aR\leq\aD^{(j)}$, {the limit process is the same for all three growth regimes} and the
distinction 
among the growth regimes is {irrelevant} \citep[Theorem 4]{kaj:taqqu:2008}.)

\begin{enumerate}[(a)]

\item Under fast-growth, we distinguish two subcases:

\begin{enumerate}[(i)]
\item If $E[( R_{1}^{(1)} )^{2}]<\infty$,
\begin{equation*}\label{eq:limfgrowth1}
A_{cs}^{(j)}(\cdot)\fidi E[( R_{1}^{(1)}
)^{2}]^{1/2}\sigma_{B_{H^{(j)}}(1)}^{(j)} B_{H^{(j)}}(\cdot),\quad
T\to\infty, 
\end{equation*}
where
\begin{equation*}\label{eq:eq:afgrowth}
a^{(j)}(T)=[\l^{(j)}T^{3}\tFD^{(j)}(T)]^{1/2},
\end{equation*}
\begin{equation*}
\sigma_{B_{H^{(j)}}(1)}^{(j)}=\frac{2}{(\aD^{(j)}-1)(2-\aD^{(j)})(3-\aD^{(j)})},
\end{equation*}
and $B_{H^{(j)}}$ is a fractional Brownian motion with Hurst exponent
\begin{equation*}
H^{(j)}=(3-\aD^{(j)})/2\in(1/2,1).
\end{equation*}

\item If $\tFR\in RV_{-\aR}, 1<{\alpha_D^{(j)}<} \aR<2$, then
\begin{equation*}
A_{cs}^{(j)}(\cdot)\fidi Z_{\aD^{(j)},\aR}(\cdot),\quad T\to\infty,
\end{equation*}
where
\begin{equation*}
a^{(j)}(T)=Tb_{R}(\l T\tFD^{(j)}(T)),
\end{equation*}
\begin{align*}
Z_{\aD^{(j)}, {\alpha_R}}(t)&=\int_{-\infty}^{\infty}\int_{0}^{\infty}\int_{0}^{\infty}rL_{t}(s,u)\cprm{N^{\infty}_{\aD^{(j)},\aR}}(ds,du),\notag\\
&\stackrel{d}{=}\left(\left(-\cos\frac{\pi\aD^{(j)}}{2}\right)\frac{2\Gamma(2-\aD^{(j)})}{\aD^{(j)}(\aD^{(j)}-1)}\right)^{1/\aD^{(j)}}\int_{-\infty}^{\infty}\int_{0}^{\infty}L_{t}(s,u)M_{\aR,m}(ds,du),\label{eq:limfgrowth2}
\end{align*}
and $M_{\aR,m}(ds,du)$ is a $\aR$-stable random measure with control measure
\begin{equation*}
m(ds,du)=ds \cdot\aD^{(j)} u^{-(\aD^{(j)}+1)}du.
\end{equation*}
Thus, the process $Z_{\aD^{(j)},\aR}(t)$ is $\aR$-stable and $H^{(j)}$-similar with
\begin{equation*}
H^{(j)}=(\aR+1-\aD^{(j)})/\aR\in(1/\aR,1).
\end{equation*}

\end{enumerate}

\item Under moderate-growth
\begin{equation*}\label{eq:mgrowthlimit}
A_{cs}^{(j)}(\cdot)\fidi c_{j}Y_{\aD^{(j)}}(\cdot/c_{j}),\quad T\to\infty,
\end{equation*}
where
\begin{equation*}
a^{(j)}(T)=T,
\end{equation*}
and
\begin{equation*}
Y_{\aD^{(j)}}(t)=\int_{-\infty}^{\infty}\int_{0}^{\infty}\int_{0}^{\infty}rL_{t}(s,u)\cprm{N^{\infty}_{\aD^{(j)},\FR}}(ds,du,dr).
\end{equation*}

\item Under slow-growth
\begin{equation*}
A_{cs}^{(j)}(\cdot)\fidi E[( R_{1}^{(1)} )^{\aD^{(j)}}]^{1/\aD^{(j)}}\Lambda_{\aD^{(j)}}(\cdot),\quad T\to\infty,
\end{equation*}
where
\begin{equation*}
a^{(j)}(T)=b_{D}^{(j)}(\l^{(j)}T),
\end{equation*}
$\Lambda_{\aD^{(j)}}$ is an $\aD^{(j)}$-stable L\'evy motion totally skewed to the right, which we can write as
\begin{align*}
E[( R_{1}^{(1)} )^{\aD^{(j)}}]^{1/\aD^{(j)}}\Lambda_{\aD^{(j)}}(t)&=\int_{-\infty}^{\infty}\int_{0}^{\infty}\int_{0}^{\infty}ur1_{\{0<s<t\}}\cprm{N^{\infty}_{\aD^{(j)},\FR}}(ds,du,dr)\\
&\stackrel{d}{=}\left(\left(-\cos\frac{\pi\aD^{(j)}}{2}\right)\frac{2\Gamma(2-\aD^{(j)})}{\aD^{(j)}(\aD^{(j)}-1)}\right)^{1/\aD^{(j)}}\int_{-\infty}^{\infty}\int_{0}^{\infty}r1_{\{0<s<t\}}M_{\aD^{(j)},m}(ds,dr),
\end{align*}
and $M_{\aD^{(j)},m}(ds,dr)$ is an $\aD^{(j)}$-stable random measure with control measure 
\begin{equation*}
m(ds,dr)=ds\FR(dr).
\end{equation*}
\end{enumerate}
\end{thm}

Real network traffic consists of several distinct types {and} in
this paper we are interested in the centered and scaled cumulative
input of the {superimposed}  stream{s} in $[0,Tt]$, namely 
\begin{equation}\label{eq:csaginp}
A_{cs}(t):=\frac{A(Tt)-\l\mD \mR Tt}{a(T)},
\end{equation}
for a suitable $a(T)$. In order to study the limit distribution of $A_{cs}(t)$ as $T\to\infty$, let $\Fast$, $\Moderate$, $\Slow$ be the subsets of indices of streams whose arrival intensities behave under the fast-, moderate-, and slow-growth regimes, respectively.

Assuming that all indices belong to one of these three classes, consider the following scenarios.

\begin{description}
\item[Scenario $\Fast$] There is at least one stream whose arrival
  intensity {satisfies} fast-growth; i.e. $\Fast  \neq \emptyset$. In this
  case, the aggregated stream's arrival intensity also {satisfies}
  fast-growth:
\begin{equation}\label{eq:sfast}
\l T \tFD(T)\geq\sum_{j\in\Fast}\l^{(j)}T\tFD^{(j)}(T)\to\infty,\quad T\to\infty.
\end{equation}

\item[Scenario $\Moderate$] No stream's arrival intensity {satisfies}
  fast-growth, but  at least one stream {satisfies}
  moderate-growth; 
i.e. $\Fast=\emptyset$ and
  $\Moderate\not=\emptyset$. Then, the aggregated stream's arrival
  intensity {satisfies} moderate growth, since 
\begin{equation}\label{eq:smoderate}
\l T\tFD(T)\to c^{\aD-1},\quad T\to\infty,
\end{equation}
where
\begin{equation}\label{eq:minad}
\aD:=\bigwedge_{j=1}^{p}\aD^{(j)}
\end{equation}
and
\begin{equation}\label{eq:cag}
c=\left(\sum_{j\in\Moderate}c_{j}^{\aD^{(j)}-1}\right)^{1/(\aD-1)}.
\end{equation}

\item[Scenario $\Slow$] All the stream's arrival intensities satisfy slow growth, that is $\Slow=\{1,\ldots,p\}$. In this case, the aggregated stream's arrival intensity also satisfies slow-growth:
\begin{equation}\label{eq:sslow}
\l T\tFD(T)=\sum_{j\in\Slow}\l^{(j)}T\tFD^{(j)}(T)\to0,\quad T\to\infty.
\end{equation}

\end{description}

The different growth regimes in Theorem \ref{thm:mikosch} are specified by the arrival intensity $\l^{(j)}$, and the distribution $\FD^{(j)}$ of the duration of the sessions of the $j$th stream. While $\l^{(j)}=\l^{(j)}(T)\to\infty$ as $T\to\infty$, $\FD^{(j)}$ does not vary with $T$. However, the growth regimes described in Scenarios $\Fast$, $\Moderate$ and $\Slow$ are given in terms of the arrival rate $\l$, and the distribution $\FD$ of the duration of the sessions of the aggregated stream and here both $\l$ and $\FD$ vary with $T$, as seen in \eqref{eq:fmix}. Therefore, we cannot directly apply Theorem \ref{thm:mikosch} for the aggregated stream when
\begin{equation}\label{eq:propj}
\l^{(j)}/\l = \textrm{proportion of the sessions that belong to the $j$th stream},\quad j=1,\ldots,p,
\end{equation}
are functions of $T$. Nevertheless, in the special case that these proportions are constant, $\FD$ does not vary with $T$, and a direct application of Theorem \ref{thm:mikosch} yields the following result.

\begin{mycor}\label{cor:constprop}
Suppose that for all $T$ (or at least for $T$ large enough), the proportions $\l^{(j)}/\l$ remain constant, $j=1,\ldots,p$, so that
\begin{equation*}
\tFD=\sum_{j=1}^{p} (\l^{(j)}/\l)\tFD^{(j)}\in RV_{-\aD},
\end{equation*}
where $\aD$ is given in \eqref{eq:minad}. Let the Scenarios $\Fast$, $\Moderate$ and $\Slow$ take the place of the fast-, moderate- and slow-growth regimes. 

If $\tFR\in RV_{-\aR}, \aR>\aD$ or $E[( R_{1}^{(1)} )^{2}]<\infty$,
then Theorem \ref{thm:mikosch} holds for $A_{cs}(\cdot)$, where
$\aD^{(j)}$, $\FD^{(j)}$ and $c_{j}$ are {replaced}  by $\aD$,
$\FD$ and the constant $c$ in \eqref{eq:cag}, respectively. 

If $\tFR\in RV_{-\aR}, \aR\leq\aD$, the distinction among Scenarios $\Fast$, $\Moderate$ and $\Slow$ is irrelevant, and limit results are discussed in Section \ref{sec:extensions}.
\end{mycor}

\paragraph{\bf{Implications of Corollary \ref{cor:constprop}.}} This result provides a partial answer to the question of how much
aggregation is required for traffic to be Gaussian at large time
scales: Suppose that at least one traffic stream falls {in} the
fast-growth regime, thus generating a cumulative input that can be
approximated by fractional Brownian motion. {When applicable,}  Corollary
\ref{cor:constprop} {implies that} the superimposed traffic load
{can also be approximated by} fractional Brownian motion.

In the case that the traffic also contains streams that satisfy the
slow-growth regime, Corollary \ref{cor:constprop} is {somewhat}
counterintuitive due to the nature of the distribution tails of the
two limit processes. {Although}
these slow-growth streams 
produce cumulative inputs that are {approximately stable}
L\'evy-motion  when  
considered individually,  with the inclusion of one single
stream that behaves under the fast-growth regime, the cumulative
aggregated input is approximately Gaussian. 

Moreover, a sufficient condition for the fast-growth regime {of}
Scenario $\Fast$ is that a single stream, say the $j$th one, satisfies
fast-growth, even if all the other streams' arrival
intensities do not follow a growth regime at all. In this sense, Scenario
$\Fast$ is a {robust assumption}. We will see that  as long as one $\a_{D}^{(j)}<\aR$, the
limit result of Corollary \ref{cor:constprop} is still valid.  

In real networks,  there {are arguably}
streams with large
initiation rates. {For instance}, the arrival rates of
 http traffic must be large, since there are a
large number of users {constantly}
accessing websites {and this} translates
into Scenario $\Fast$.
Furthermore,
even though {some studies report or assume 
 session transmission rates
have infinite variance}, the assumption $E[(R_{1}^{(1)})^{2}]<\infty$
may be justified { by rate constraint mechanisms required for
  congestion control.}
{Although assumptions always deserve rigorous scrutiny,}
 Corollary \ref{cor:constprop} provides a
compelling explanation for the data example in Section
\ref{sec:intro}. 

{We now address more general assumptions which
allow the conclusions} of  Corollary \ref{cor:constprop} to hold. While the assumption of constant proportions $\l^{(j)}/\l$ {may
sometimes be}
reasonable, in general the proportions of sessions corresponding to {the}
 $p$ independent streams are not constant over time.
We may have that $\lim\l^{(j)}/\l$ exists or, more
generally, that $\l^{(j)}/\l\in(a,b)\subset(0,1)$ varies with no limit
whatsoever. Extending the conclusions of Corollary
\ref{cor:constprop} to under 
 weaker assumptions is the focus of the next section. 

\section{Behavior of cumulative load of  aggregated streams}\label{sec:main}

We prove that the conclusion of Corollary \ref{cor:constprop} is
still valid even when the proportion of the sessions corresponding to
the $p$ independent streams is not constant.
 Here is the result:

\begin{thm}\label{thm:novanishprop}
Assume that
\begin{equation}\label{eq:novanishprop}
\liminf_{T\to\infty}\bigvee_{j:\aD^{(j)}=\aD} \l^{(j)}/\l >0.
\end{equation}
Then, the conclusions of Corollary \ref{cor:constprop} regarding the limit distribution of the cumulative input of the aggregated stream $A_{cs}(\cdot)$ are still valid.
\end{thm}

 Condition \ref{eq:novanishprop} implies that there exists $d>0$ such that for all $T$ sufficiently large, there is at least one $k=k(T)$ such that $\aD^{(k)}=\aD$ and $\l^{(k)}/\l>d$. Roughly speaking, this means that the proportion of the traffic with the heaviest-tailed duration always remains greater than a positive quantity.

All the limits in Theorem \ref{thm:novanishprop} follow from the convergence of the characteristic function of the finite-dimensional distributions (\emph{fidi chf}) of the processes. Thus, let $m\geq1$ represent the dimension, $0\leq t_{1},\ldots,t_{m}$ the times, and $z_{1},\ldots,z_{m}$ arbitrary real numbers; we need
\begin{equation*}
g(s,u,r)=\exp\left\{i\sum_{j=1}^{m}z_{j}rL_{t_{j}}(s,u)\right\}-1-i\sum_{j=1}^{m}z_{j}rL_{t_{j}}(s,u),
\end{equation*}
as defined in Proposition \ref{prop:chf}.

From the second integral in \eqref{eq:length}, we can compute the partial derivative of $L_{t}(s,u)$ with respect to $u$, which yields
\begin{equation}\label{eq:gufun}
g_{u}(s-u,u,r):=
\frac{\partial}{\partial u} g_{\big \vert_{(s-u,u,r)}}=
i\left(\exp\left\{i\sum_{j=1}^{m}z_{j}rL_{t_{j}}(s-u,u)\right\}-1\right)\sum_{k=1}^{m}z_{k}r1_{[0,t_{k}]}(s),
\end{equation}
where $g_{u}$ is the partial derivative of $g(s,u,r)$ with respect to
$u$. Moreover, putting together \eqref{eqn:sid4}, \eqref{eqn:sid1}, \eqref{eqn:sid3},
the bounds in Lemmas
\ref{lem:lengthprop} and  \ref{lem:ebounds}, {we get}
\begin{equation}\label{eq:ebound}
\left|\exp\left\{i\sum_{j=1}^{m}z_{j}rL_{t_{j}}(s-u,u)\right\}-1\right|\leq 2\sum_{j=1}^{m}|z_{j}|^{\z}(t_{j}\wedge u)^{\z}r^{\z},\quad 0\leq\z\leq1.
\end{equation}

We will use three more relations in the proof of Theorem \ref{thm:novanishprop}: For $0<\eta<1$, there exists a number $T_{0}=T_{0}(\eta)>0$ such that for $T\geq T_{0}$ and $b_{D}(\l T)\geq T_{0}$,
\begin{equation}\label{eq:potterbound}
2u^{-\aD}\left\{u^{-\eta}\vee u^{\eta}\right\}\geq\begin{cases}
\tFD(Tu)/\tFD(T),& u\geq T_{0}/T,\\
\tFD(b_{D}(\l T)u)/\tFD(b_{D}(\l T)),& u\geq T_{0}/b(\l T),
\end{cases}
\end{equation}
\begin{equation}\label{eq:fmarkovbound}
\frac{\tFD(Tu)}{\tFD(T)}\leq \mD T^{\aD-1+\eta}u^{-1},\quad u<T_{0}/T,
\end{equation}
and
\begin{equation}\label{eq:fbmarkovbound}
\frac{\tFD(b_{D}(\l T)u)}{\tFD(b_{D}(\l T))}\leq \mD b_{D}(\l T)^{\aD-1+\eta}u^{-1},\quad u<T_{0}/b_{D}(\l T).
\end{equation}
We can readily derive \eqref{eq:potterbound} from Lemma \ref{lem:fpotter}. Both \eqref{eq:fmarkovbound} and \eqref{eq:fbmarkovbound} follow from Markov's Inequality and, for example, \citet[][Proposition 0.8]{resnick:1987} or \citet[][Proposition 1.3.6]{bingham:goldie:teugels:1987}.

\begin{proof}[Proof of Theorem \ref{thm:novanishprop}]

First, we will prove parts (b) and (c). For both parts, set $0<\eta<\aD-1$ and $0<\z<1$ such that
\begin{equation*}
\aD+\eta<1+\zeta<\begin{cases}
\aR,&\textrm{ if ${\bar F}_{R}\in RV_{-\aR}, 1<\aR<2$},\\
2,&\textrm{ if $E[( R_{1}^{(1)} )^{2}]<\infty$}.
\end{cases}
\end{equation*}

\emph{Part (b).} Under Scenario $\Moderate$, use $a(T)=T$ and apply Proposition \ref{prop:chf}, yielding
\begin{equation*}
\ln E\exp\left\{i\sum_{j=1}^{m}z_{m}A_{cs}(t_{j})\right\}=\int_{-\infty}^{\infty}\int_{0}^{\infty}\int_{0}^{\infty}g_{u}(s-u,u,r)\l T\tFD(T)\frac{\tFD(Tu)}{\tFD(T)}dsdu\FR(dr),
\end{equation*}
and if we can take the limit inside the integral as $T\to\infty$, performing afterwards an integration by parts in $u$ gives 
\begin{align}\label{eq:limm}
&\to\int_{-\infty}^{\infty}\int_{0}^{\infty}\int_{0}^{\infty}g_{u}(s-u,u,r)c^{\aD-1}u^{-\aD}dsdu\FR(dr)\notag\\
&=c^{\aD-1}\int_{-\infty}^{\infty}\int_{0}^{\infty}\int_{0}^{\infty}g(s,u,r)EN^{\infty}_{\aD^{(j)},\FR}(ds,du,dr),
\end{align}
which is the log \fidichf of $cY_{\aD}(\cdot/c)$. Thus, it suffices to justify taking the limit inside the integral.

First observe there exists a number $T_{0}>0$ such that for $T\geq T_{0}$, 
\begin{equation}\label{eq:boundm}
\l T\tFD(T)\leq c^{\aD-1} + \eta,
\end{equation}
by the moderate-growth assumption. Together with \eqref{eq:ebound}-\eqref{eq:fmarkovbound} and a possibly larger $T_{0}$, the above implies that the integrand in the left side of \eqref{eq:limm} is bounded in $\{u\geq T_{0}/T\}$ by
\begin{equation*}
B_{\Moderate,(>)}(s,u,r):=4\left(c^{\aD-1} + \eta\right)u^{-\aD}(u^{-\eta}\vee u^{\eta})\sum_{j=1}^{m}\sum_{k=1}^{m}|z_{j}|^{\z}|z_{k}|(t_{j}\wedge u)^{\z}r^{1+\z}1_{[0,t_{k}]}(s),
\end{equation*}
and bounded in $\{u< T_{0}/T\}$ by
\begin{equation*}
B_{\Moderate,(<)}(s,u,r):=2\left(c^{\aD-1} + \eta\right)T_{0}^{\aD-1+\eta}\mD\sum_{j=1}^{m}\sum_{k=1}^{m}|z_{j}|^{\z}|z_{k}|u^{\z-\aD-\eta}r^{1+\z}1_{[0,t_{k}]}(s)1_{(0,1)}(u),
\end{equation*}
whenever $T\geq T_{0}$.  Here we used the bound
\begin{equation}\label{eq:ubound}
u^{\z}\leq(T_{0}/T)^{\aD-1+\eta}u^{1+\z-\aD-\eta},\quad 0<u<T_{0}/T.
\end{equation}

Therefore, \eqref{eq:limm} follows by the dominated convergence
theorem, since for all $T\geq T_{0}$ 
\begin{align*}
\int_{-\infty}^{\infty}\int_{0}^{\infty}\int_{0}^{\infty} &B_{\Moderate,(>)}(s,u,r)dsdu\FR(dr)\\
&\leq 4\left(c^{\aD-1} + \eta\right)E[(R_{1}^{(1)})^{1+\z}]\sum_{k=1}^{m}\sum_{j=1}^{m}|z_{j}|^{\z}|z_{k}|t_{k}\bigg\{ \int_{0}^{1}u^{\z-\aD-\eta}du+ t_{j}^{\z}\int_{1}^{\infty}u^{-\aD+\eta}du \bigg\},
\end{align*}
and
\begin{align*}
\int_{-\infty}^{\infty}\int_{0}^{\infty}\int_{0}^{\infty} &B_{\Moderate,(<)}(s,u,r)dsdu\FR(dr)\\
&\leq2\left(c^{\aD-1} + \eta\right)T_{0}^{\aD-1+\eta}\mD\sum_{j=1}^{m}\sum_{k=1}^{m}|z_{j}|^{\z}|z_{k}|t_{k}E[(R_{1}^{(1)})^{1+\z}]\int_{0}^{1}u^{\z-\aD-\eta}du,
\end{align*}
which are both finite by our choice of $\eta$ and $\z$.

\emph{Part (c).} Under Scenario $\Slow$, $a(T)=b_{D}(\l T)$, so we use Lemma \ref{lem:bbounds} and \citet[][Lemma 1]{mikosch:resnick:rootzen:stegeman:2002} to get
\begin{equation*}
\lim_{T\to\infty}T/a(T)=\infty.
\end{equation*}
Thus, it follows from the definition of $L_{t}(s,u)$ in \eqref{eq:length} that
\begin{equation*}
\lim_{T\to\infty}L_{tT/a(T)}(sT/a(T)-u,u)=u1_{[0,t]}(s).
\end{equation*}

Now, apply Proposition \ref{prop:chf}, perform the change of variables $r\mapsto ra(T)/T$, $u\mapsto uT/a(T)$, and use the scaling property in Lemma \ref{lem:lengthprop} to get
\begin{align*}
&\ln E\exp\left\{i\sum_{j=1}^{m}z_{m}A_{cs}(t_{j})\right\}\notag\\
&=\int_{-\infty}^{\infty}\int_{0}^{\infty}\int_{0}^{\infty}g_{u}(s-ua(T)/T,ua(T)/T,rT/a(T))\l a(T)\tFD(a(T)u)dsdu\FR(dr)\notag\\
&=\int_{-\infty}^{\infty}\int_{0}^{\infty}\int_{0}^{\infty}i\left(\exp\left\{\sum_{j=1}^{m}z_{j}rL_{t_{j}T/a(T)}(sT/a(T)-u,u)\right\}-1\right)\notag\\
&\quad\quad\quad\quad\quad\quad\quad\quad\quad\quad\quad\quad\quad\quad\quad\quad\times\sum_{k=1}^{m}z_{k}r1_{(0,t_{k})}(s)\l T\tFD(a(T)u)dsdu\FR(dr),\notag\\
\end{align*}
and {assuming we can take the limit inside the integral, the
limit as $T\to\infty$ is}
\begin{equation}\label{eq:lims}
\int_{-\infty}^{\infty}\int_{0}^{\infty}\int_{0}^{\infty}i\left(\exp\left\{\sum_{j=1}^{m}z_{j}ur1_{(0,t_{j})}(s)\right\}-1\right)\sum_{k=1}^{m}z_{k}r1_{(0,t_{k})}(s)u^{-\aD}dsdu\FR(dr),
\end{equation} 
which is the log \fidichf of $E[( R_{1}^{(1)} )^{\aD}]^{1/\aD}\Lambda_{a_{D}}(\cdot)$. Therefore, we must justify passing the limit inside the integral. This is done as follows.

First, by Lemma \ref{lem:frv}, there exists a number $T_{0}>0$ such that for $T\geq T_{0}$, 
\begin{equation}\label{eq:bounds}
\l T\tFD(a(T))\leq 2.
\end{equation}

Hence, by taking a possibly larger $T_{0}$, \eqref{eq:ebound}, \eqref{eq:potterbound} and \eqref{eq:fbmarkovbound} imply that the integrand in the left side of \eqref{eq:lims} is bounded in $\{u\geq T_0/a(T)\}$ by
\begin{equation*}
B_{\Slow,(>)}(s,u,r):=8u^{-\aD}(u^{-\eta}\vee u^{\eta})\sum_{j=1}^{m}\sum_{k=1}^{m}|z_{j}|^{\z}|z_{k}|(t_{j}\wedge u)^{\z}r^{1+\z}1_{[0,t_{k}]}(s),
\end{equation*}
and bounded in $\{u< T_0/a(T)\}$ by
\begin{equation*}
B_{\Slow,(<)}(s,u,r):=4T_0^{\aD-1+\eta}\mD\sum_{j=1}^{m}\sum_{k=1}^{m}|z_{j}|^{\z}|z_{k}|u^{\z-\aD-\eta}r^{1+\z}1_{[0,t_{k}]}(s)1_{(0,1)}(u),
\end{equation*}
whenever $T\geq T_{0}$ and $a(T)\geq T_0$, using
\begin{equation*}
u^{\z}\leq(T_{0}/a(T))^{\aD-1+\eta}u^{1+\z-\aD-\eta},\quad 0<u<T_{0}/a(T).
\end{equation*}

Therefore, \eqref{eq:lims} follows exactly as in part (b) from the dominated convergence theorem.

\emph{Part (a).} Under Scenario $\Fast$ and $E[(R_{1}^{(1)})^{2}]<\infty$, set $a(T)=[\l T^{3}\tFD(T)]^{1/2}$. Use Proposition \ref{prop:chf} and the change of variables $r\mapsto r a(T)/T$, to write
\begin{equation}\label{eq:chfffin}
\ln E\exp\left\{i\sum_{j=1}^{m}z_{m}A_{cs}(t_{j})\right\}=\int_{-\infty}^{\infty}\int_{0}^{\infty}\int_{0}^{\infty} g_{u}(s-u,u,rT/a(T))(a(T)/T)^{2}\frac{\tFD(Tu)}{\tFD(T)}dsdu\FR(dr),
\end{equation}
where
\begin{equation*}
(a(T)/T)^{2}=\l T\tFD(T)\to\infty,\quad T\to\infty,
\end{equation*}
by the fast-growth assumption.

By \eqref{eq:gufun}, as $T\to\infty$
\begin{equation*}
g_{u}(s-u,u,rT/a(T))(a(T)/T)^{2}= i\bigg(i\sum_{j=1}^{m}z_{j}r\frac{T}{a(T)}L_{t_{j}}(s-u,u)+o\left(\frac{T}{a(T)}\right)\bigg)\sum_{k=1}z_{k}r1_{[0,t_{k}]}(s)\frac{a(T)}{T}.
\end{equation*}
Hence, assuming we can pass the limit inside the integral, we use Lemma \ref{lem:lengthprop} (iii) to write
\begin{align*}
&\lim_{T\to\infty}\ln E\exp\{i\sum_{j=1}^{m}z_{m}A_{cs}(t_{j})\}\\
&=-\sum_{j=1}^{m}\sum_{k=1}^{m}z_{j}z_{k}\int_{-\infty}^{\infty}\int_{0}^{\infty}\int_{0}^{\infty}r^{2}L_{t_{j}}(s-u,u)1_{[0,t_{k}]}(s)u^{-\aD}dsdu\FR(dr)\\
&=-E[(R_{1}^{(1)})^{2}]\bigg(\frac{1}{(\aD-1)(2-\aD)(3-\aD)}\bigg)\times\\
&\bigg\{\sum_{j=1}^{m}\sum_{k=1}^{j}z_{j}z_{k}t_{k}^{3-\aD}+\sum_{j=1}^{m}\sum_{k=j+1}^{m}z_{j}z_{k}\left( t_{k}^{3-\aD}-(t_{k}-t_{j})^{3-\aD} \right)\bigg\}\\
&=-\frac12 E[(R_{1}^{(1)})^{2}]\sigma_{B_{H}(1)}^{2}\sum_{j=1}^{m}\sum_{k=1}^{m}z_{j}z_{k}\frac 12\left\{ |t_{j}|^{2H}+|t_{k}|^{2H} - |t_{j}-t_{k}|^{2H} \right\},
\end{align*}
where the last line follows by rearranging of the terms in the sum, $\sigma_{B_{H}(1)}^{2}$ is given in \eqref{eq:varbh1} and $H=(3-\aD)/2$. It remains to prove that we can take the limit inside the integral.

Let $0<\eta<\aD-1$. We use \eqref{eq:ebound}-\eqref{eq:fmarkovbound} with $\z=1$ and a possibly larger $T_{0}$, which imply that the integrand in Eq. \ref{eq:chfffin} is bounded in $\{u\geq T_{0}/T\}$ by
\begin{equation*}
B_{\Fast,(>)}:=4u^{-\aD}(u^{-\eta}\vee u^{\eta})\sum_{j=1}^{m}\sum_{k=1}^{m}|z_{j}z_{k}|(t_{j}\wedge u)r^{2}1_{[0,t_{k}]}(s),
\end{equation*}
and bounded in $\{u< T_{0}/T\}$ by
\begin{equation*}
B_{\Fast,(<)}:=2T_{0}^{\aD-1+\eta}\mD\sum_{j=1}^{m}\sum_{k=1}^{m}|z_{j}z_{k}|u^{1-\aD-\eta}r^{2}1_{[0,t_{k}]}(s)1_{(0,1)}(u),
\end{equation*}
whenever $T\geq T_{0}$. Here we used
\begin{equation*}
u\leq(T_{0}/T)^{\aD-1+\eta}u^{2-\aD-\eta},\quad 0<u<T_{0}/T.
\end{equation*}
The result now follows by the dominated convergence theorem, since
\begin{align*}
\int_{-\infty}^{\infty}\int_{0}^{\infty}\int_{0}^{\infty}&B_{\Fast,(>)}(s,u,r)dsdu\FR(dr)\\
&\leq
4E[(R_{1}^{(1)})^{2}]\sum_{j=1}^{m}\sum_{k=1}^{m}|z_{j}z_{k}|t_{k}\bigg\{\int_{0}^{1}u^{1-\aD-\eta}du+
t_{j} ER_{1}^{(1)}\int_{1}^{\infty}u^{-\aD+\eta}du \bigg\}, \\
\intertext{and}
\int_{-\infty}^{\infty}\int_{0}^{\infty}\int_{0}^{\infty}&B_{\Fast,(<)}(s,u,r)dsdu\FR(dr)\\
&\leq2 T_{0}^{\aD-1+\eta}\mD\sum_{j=1}^{m}\sum_{k=1}^{m}|z_{j}z_{k}|t_{k}E[(R_{1}^{(1)})^{2}]\int_{0}^{1}u^{1-\aD-\eta}du,
\end{align*}
and both bounds are finite by our choice of $\eta$.

Finally, still under Scenario $\Fast$, assume $\tFR\in RV_{-\aR}, 1<\aR<2$. Set $a(T)=Tb_{R}(\l T\tFD(T))$. By Proposition \ref{prop:chf}, an integration by parts in $u$ and the change of variables $s\mapsto s+u$:
\begin{align}
\ln &E\exp\left\{i\sum_{j=1}^{m}z_{m}A_{cs}(t_{j})\right\}\notag\\
&=\l T\tFD(T)\tFR(b_{R}(\l T\tFD(T)))\int_{-\infty}^{\infty}\int_{0}^{\infty}\int_{0}^{\infty}g(s,u,r)ds\frac{\FD(Tdu)}{\tFD(T)}\frac{\FR(b_{R}(\l T\tFD(T))dr)}{\tFR(b_{R}(\l T\tFD(T)))}\notag\\
&=\l T\tFD(T)\tFR(b_{R}(\l T\tFD(T)))\left\{I_{(u>\e,r>\e)}+I_{(u<\e,r>\e)}+I_{(r<\e)}\right\}\label{eq:chffinf},
\end{align}
where $\e>0$ and we split the integral into three parts according to the domains of integration $\{u>\e,r>\e\}$, $\{u<\e,r>\e\}$ and $\{r<\e\}$, respectively. To establish the limit result, we will take $\lim_{\e\to0}\lim_{T\to\infty}$ on both sides of \eqref{eq:chffinf}.

Fix $\e>0$ and start with the first integral. Let
\begin{align*}
\nu_{T}(du,dr) &:= \left(u\frac{\FD(Tdu)}{\tFD(T)}\right)\left(r\frac{\FR(b_{R}(\l T\tFD(T))dr)}{\tFR(b_{R}(\l T\tFD(T)))}\right),\\
\nu(du,dr) &:= \aD u^{-\aD} {du}\,\aR r^{-\aR} {dr},\\
G(u,r) &:= \frac{1}{ur} \int_{-\infty}^{\infty}g(s,u,r)ds,
\end{align*}
which allows writing
\begin{equation*}
I_{(u>\e,r>\e)}=\int_{\e}^{\infty}\int_{\e}^{\infty}G(u,r)\nu_{T}(du,dr).
\end{equation*}

The fast-growth regime, regular variation of $\tFR$, 
{\eqref{eqn:sid4}}
 and \citet[][Theorem 2.8]{billingsley:1999} imply $\nu_{T}\xrightarrow{v}\nu$ as $T\to\infty$. Moreover, $G(u,r)$ is jointly continuous and it follows from Lemmas \ref{lem:lengthprop} and \ref{lem:ebounds} that $|G(u,r)|\leq d_{0}\sum_{j=1}^{m}|z_{j}|t_{j}<\infty$, where $d_{0}$ is a positive constant. Therefore:
\begin{align}
\lim_{T\to\infty}I_{(u>\e,r>\e)} &= \int_{\e}^{\infty}\int_{\e}^{\infty}G(u,r)\nu(du,dr)\notag\\
&=\int_{-\infty}^{\infty}\int_{\e}^{\infty}\int_{\e}^{\infty}g(s,u,r)ds\cdot \aD u^{-(\aD+1)}du\cdot \aR r^{-(\aR+1)}dr\label{eq:lim1chf}
\end{align}

Now let $0\leq\z,\eta\leq1$ such that $\aD+\eta<1+\z<\aR$. By Lemmas \ref{lem:lengthprop} and \ref{lem:ebounds}, there exists $d_{\z}>0$ such that
\begin{equation*}
\left|I_{(u<\e,r>\e)}\right|\leq d_{\z}\sum_{j=1}^{m}|z_{j}|t_{j}\int_{\e}^{\infty}r^{1+\z}\frac{\FR(b_{R}(\l T\tFD(T))dr)}{\tFR(b_{R}(\l T\tFD(T)))}\int_{0}^{\e}u^{1+\z}\frac{\FD(Tdu)}{\tFD(T)}.
\end{equation*}
Furthermore, by fast-growth and regular variation of $\tFR$
\begin{equation*}
\int_{\e}^{\infty}r^{1+\z}\frac{\FR(b_{R}(\l T\tFD(T))dr)}{\tFR(b_{R}(\l T\tFD(T)))}\to\aR\int_{\e}^{\infty}r^{\z-\aR}dr=\frac{\aR}{\aR-1-\z}\e^{1+\z-\aR},\quad\quad T\to\infty.
\end{equation*}
Similarly, integration by parts and \eqref{eq:fmarkovbound} with $T\geq T_{0}$ such that $\e<T_{0}/T$ yields
\begin{align*}
\int_{0}^{\e}u^{1+\z}\frac{\FD(Tdu)}{\tFD(T)}&=(1+\z)\int_{0}^{\e}u^{\z}\frac{\tFD(Tu)}{\tFD(T)}du\\
&\leq (1+\z)\mD T_{0}^{\aD-1+\eta}\int_{0}^{\e}u^{\z-\aD-\eta}du\\
&\leq \frac{(1+\z)\mD T_{0}^{\aD-1+\eta}}{1+\z-\aD-\eta}\e^{1+\z-\aD-\eta},
\end{align*}
where we used the bound \eqref{eq:ubound}. Thus
\begin{equation}\label{eq:lim2chf}
\limsup_{T\to\infty}\left|I_{(u<\e,r>\e)}\right|\leq constant\cdot\e^{2(1+\z)-\aR-\aD-\z},
\end{equation}
where the exponent of $\e$ is positive if we additionally let $(\aR+\aD+\eta)/2<1+\z$.

For $I_{(r<\e)}$, use again Lemmas \ref{lem:lengthprop} and \ref{lem:ebounds} to get a $d_{1}>0$ such that
\begin{align*}
|I_{(r<\e)}|\leq d_{1}\sum_{j=1}^{m}\sum_{k=1}^{m}|z_{j}z_{k}|t_{j}\int_{0}^{\e}r^{2}\frac{\FR(b_{R}(\l T\tFD(T))dr)}{\tFR(b_{R}(\l T\tFD(T)))}\int_{0}^{\infty}u(t_{k}\wedge u)\frac{\FD(Tdu)}{\tFD(T)}.
\end{align*}
Analogously to the bound for $I_{(u<\e,r>\e)}$, it can be readily shown that there exists $T_{0}>0$ such that for $T\geq T_{0}$
\begin{equation*}
\int_{0}^{\e}r^{2}\frac{\FR(b_{R}(\l T\tFD(T))dr)}{\tFR(b_{R}(\l T\tFD(T)))}\leq \frac{3\mR T_{0}^{\aR}}{2-\aR}\epsilon^{2-\aR},
\end{equation*}
and
\begin{equation*}
\int_{0}^{\infty}u(t_{k}\wedge u)\frac{\FD(Tdu)}{\tFD(T)}\leq\int_{0}^{1}u^{2}\frac{\FD(Tdu)}{\tFD(T)}+t_{k}\int_{1}^{\infty}u\frac{\FD(Tdu)}{\tFD(T)}\leq \frac{3\mD T_{0}^{\aD}}{2-\aD}+t_{k}\frac{\aD}{\aD-1},
\end{equation*}
whence
\begin{equation}\label{eq:lim3chf}
\limsup_{T\to\infty}\left|I_{(r<\e)}\right|\leq constant\cdot\e^{2-\aR}.
\end{equation}

Also, by fast growth and the regular variation of $\tFR$
\begin{equation}\label{eq:frrv}
\l T\tFD(T)\tFR(b_{R}(\l T\tFD(T)))\to1,\quad\quad T\to\infty.
\end{equation}

Finally, we can put together \eqref{eq:chffinf}-\eqref{eq:frrv} to write:
\begin{align*}
\lim_{T\to\infty}\ln E\exp\left\{i\sum_{j=1}^{m}z_{j}A_{cs}(t_{j})\right\}&=\lim_{\e\to0}\lim_{T\to\infty}\ln E\exp\left\{i\sum_{j=1}^{m}z_{j}A_{cs}(t_{j})\right\}\\
&=\int_{-\infty}^{\infty}\int_{0}^{\infty}\int_{0}^{\infty}g(s,u,r)ds \cdot\aD u^{-(\aD+1)}du\cdot\aR r^{-(\aR+1)}dr.
\end{align*}
\end{proof}

\section{Remaining choices of $\aD$ and $\aR$}\label{sec:extensions}

We first study what happens if $\aR<\aD$, namely, $\aR<\aD^{(j)}$ for $j=1,\ldots,p$. Consider the centered and scaled cumulative input of any of the streams, say the first one, for simplicity. Set the normalizing term to $a^{(1)}(T)=b_{R}(\l T)$ (and the same for all other streams). Write
\begin{align*}
A_{cs}^{(1)}(t)&=\frac{1}{b_{R}(\l T)}\int_{-\infty}^{\infty}\int_{0}^{\infty}\int_{0}^{\infty}rL_{Tt}(s,u)\cprm{N}^{(1)}(ds,du,dr)\\
&=\int_{-\infty}^{\infty}\int_{0}^{\infty}\int_{0}^{\infty}rL_{Tt}(Ts,u)\cprm{N}^{(1)}(Tds,du,b_{R}(\l T)dr).
\end{align*}

First, note from the definition of $L_{t}(s,u)$ in \eqref{eq:length} that
\begin{equation*}
\lim_{T\to\infty}L_{Tt}(Ts,u)=u1_{[0,t]}(s).
\end{equation*}
Thus, analogous to the proof of Proposition \ref{prop:chf}, it can be shown that the log \fidichf of $A_{cs}^{(1)}$ is
\begin{align}
&\ln E \exp \left\{ i\sum_{j=1}^{m}z_{j}A_{cs}^{(1)}(t_{j}) \right\}\notag\\
&=\frac{\l^{(1)}}{\l} \l T\tFR(b_{R}(\l T))\notag\\
&\times\int_{-\infty}^{\infty}\int_{0}^{\infty}\int_{0}^{\infty}i\left(\exp\left\{ir\sum_{j=1}^{m}z_{j}L_{Tt_{j}}(Ts,u)\right\}-1\right)\sum_{k=1}^{m}z_{k}L_{Tt_{j}}(Ts,u)ds\FD^{(1)}(du)\frac{\tFR(b_{R}(\l T)r)}{\tFR(b_{R}(\l T))}dr.\label{eq:41}
\end{align}

Now observe that, provided we can take the limit inside the integral
\begin{align}
&\lim_{T\to\infty}\int_{-\infty}^{\infty}\int_{0}^{\infty}\int_{0}^{\infty}i\left(\exp\left\{ir\sum_{j=1}^{m}z_{j}L_{Tt_{j}}(Ts,u)\right\}-1\right)\sum_{k=1}^{m}z_{k}L_{Tt_{j}}(Ts,u)ds\FD^{(1)}(du)\frac{\tFR(b_{R}(\l T)r)}{\tFR(b_{R}(\l T))}dr\notag\\
&=\int_{-\infty}^{\infty}\int_{0}^{\infty}\int_{0}^{\infty}i\left(\exp\left\{ir\sum_{j=1}^{m}z_{j}u1_{[0,t_{j}]}(s)\right\}-1\right)\sum_{k=1}^{m}z_{k}u1_{[0,t_{j}]}(s)ds\FD^{(1)}(du)r^{-\aR}dr.\label{eq:42}
\end{align}
Since $|\l^{(1)}/\l|\leq1$, then
\begin{align}
\epsilon^{(1)}(T)&:=\ln E \exp \left\{ i\sum_{j=1}^{m}z_{j}A_{cs}^{(1)}(t_{j}) \right\}\notag\\
&\quad-\frac{\l^{(1)}}{\l}\int_{-\infty}^{\infty}\int_{0}^{\infty}\int_{0}^{\infty}i\left(\exp\left\{ir\sum_{j=1}^{m}z_{j}u1_{[0,t_{j}]}(s)\right\}-1\right)\sum_{k=1}^{m}z_{k}u1_{[0,t_{j}]}(s)ds\FD^{(1)}(du)r^{-\aR}dr\notag\\
&\to0,\quad\quad T\to\infty,\label{eq:43}
\end{align}
which yields the following result.

\begin{thm}\label{thm:arlessad}
Let $\Psi:=\Psi_{T}$ be the \fidichf of
\begin{align}
\sum_{j=1}^{p}\frac{\l^{(j)}}{\l}E[(D_{1}^{(j)})^{\aR}]^{1/\aR} &\Lambda_{\aR}(t)\notag\\
&\stackrel{d}{=}\left(\left(-\cos\frac{\pi\aR}{2}\right)\frac{2\Gamma(2-\aR)}{\aR(\aR-1)}\right)^{-1/\aR}\int_{-\infty}^{\infty}\int_{0}^{\infty}1_{[0,t]}(s)uM_{\aR}(ds,du),\label{eq:chfasymparlessad}
\end{align}
where for each $T$, $\Lambda_{\aR}(\cdot)$ is an  $\aR-$stable L\'evy motion totally skewed to the right with index $\aR$ and $M_{\aR}(ds,du)$ is $\aR-$stable with control measure $m(ds,du)=ds\FD(du)$. Then,
\begin{equation}\label{eq:45}
\lim_{T\to\infty}\left\{\ln E \exp \bigg\{ i\sum_{j=1}^{m}z_{j}A_{cs}^{(1)}(t_{j}) \bigg\}-\ln\Psi_{t_{1},\ldots,t_{m}}(z_{1},\ldots,z_{m})\right\}=0.
\end{equation}

In addition, if for $j=1,\ldots,p$, the limits
\begin{equation}\label{eq:limpropj}
w^{(j)}:=\lim_{T\to\infty}\l^{(j)}/\l
\end{equation}
exist, then the \fidichf of $A_{cs}(\cdot)$ converges to the \fidichf
of the process defined by \eqref{eq:chfasymparlessad}, with $w^{(j)}$ and $\sum_{j=1}^{p}w^{(j)}F_{D}^{(j)}(\cdot)$ replacing $\l^{(j)}/\l$ and $F_{D}(\cdot)$.
\end{thm}

\begin{proof}
By the independence of $N^{(j)}, j=1,\ldots,p$,
\begin{align}
\ln E \exp &\left\{ i\sum_{j=1}^{m}z_{j}A_{cs}(t_{j}) \right\}\notag\\
&-\int_{-\infty}^{\infty}\int_{0}^{\infty}\int_{0}^{\infty}i\left(\exp\left\{ir\sum_{j=1}^{m}z_{j}u1_{[0,t_{j}]}(s)\right\}-1\right)\sum_{k=1}^{m}z_{k}u1_{[0,t_{j}]}(s)ds\FD(du)r^{-\aR}dr\notag\\
&=\sum_{j=1}^{p}\epsilon^{(j)}(T)\to0\quad\quad T\to\infty.\label{eq:47}
\end{align}
Analogously to the proof of \eqref{eq:limschf}, the second integral in \eqref{eq:47} is equal to $\ln\Psi_{t_{1},\ldots,t_{m}}(z_{1},\ldots,z_{m})$. Thus, it
only remains to justify taking the limit \eqref{eq:43}. 

Let $0<\z<\z'<1$ and $0<\eta<1$ such that $1+\z<\aR-\eta<\aR+\eta<1+\z'<\aD$. Similarly to \eqref{eq:potterbound} and \eqref{eq:fbmarkovbound}, there exists $T_{0}:=T_{0}(\eta)>0$ such that for $T\geq T_{0}$ and $b_{R}(\l T)\geq T_{0}$,
\begin{equation*}
\frac{\tFR(b_{R}(\l T)r)}{\tFR(b_{R}(\l T))}\leq\begin{cases}
2r^{-\aR}\left\{r^{-\eta}\vee r^{\eta}\right\}, &r\geq T_{0}/b_{R}(\l T),\\
\mR b_{R}(\l T)^{\aR-1+\eta}r^{-1}, &r\in\R.
\end{cases}
\end{equation*}
Together with Lemmas \ref{lem:lengthprop} and \ref{lem:ebounds}, this implies that the integrand in the left side of \eqref{eq:42} is bounded in $\{r\geq 1\}$
\begin{equation*}
B_{(>)}:=2^{1-\z}u^{1+\z}r^{\z-\aR+\eta}\sum_{j=1}^{m}\sum_{k=1}^{m}|z_{j}|^{\z}|z_{k}|1_{[0,t_{k}]}(s)1_{[1,\infty)}(u),
\end{equation*}
and bounded in $\{r< 1\}$ by
\begin{equation*}
B_{(<)}:=2^{1-\z'}\mR T_{0}^{\aR-1+\eta} u^{1+\z'}r^{\z'-\aR-\eta} \sum_{j=1}^{m}\sum_{k=1}^{m}|z_{j}|^{\z'}|z_{k}|1_{[0,t_{k}]}(s)1_{(0,1)}(r),
\end{equation*}
whenever $b_{R}(\l T)>T_{0}$. Here we used
\begin{equation*}
r^{\z'}\leq(T_{0}/(b_{R}(\l T))^{\aR-1+\eta}r^{1+\z'-\aD-\eta}.
\end{equation*}
By our choice of $\z$, $\z'$ and $\eta$, both bounds are integrable and we can use dominated convergence to prove the result.
\end{proof}

In principle, it also is possible to have $\aD=\aR$. However, we
cannot say much except in the special case $\aD=\aR=2$, in which
{case} the limit process is a Brownian motion {provided
\eqref{eq:limpropj} holds}. We refer the reader to \citet[][Theorem
4]{kaj:taqqu:2008} for the formal statement of this case. 

\section{Technical proofs}\label{sec:techproofs}

This section contains a collection of technical results {needed for our
proofs.}
The first lemma establishes bounds for
$b_{D}(\cdot)=(1/\tFD)^{\leftarrow}(\cdot)$
which yield $b_{D}(\lambda T) \to \infty$.
This is not immediate since the function {$b_D$ depends on $T$.}

\begin{mylem}\label{lem:bbounds}
The quantile functions given \eqref{eq:bj} and \eqref{eq:bmix} satisfy the following inequality.
\begin{equation}\label{eq:bbounds}
\bigvee_{j=1}^{p} b_{D}^{(j)}(p\l^{(j)} T)\geq b_{D}(\l T) \geq \bigvee_{j=1}^{p} b_{D}^{(j)}(\l^{(j)} T),\quad T>0.
\end{equation}
Hence
\begin{equation*}
b_{D}(\l T)\to\infty,\quad T\to\infty.
\end{equation*}
\end{mylem}
\begin{proof}

Since $\tFD^{(j)}$ is decreasing for all $j$, then
\begin{align*}
\tFD\left(\bigvee_{j=1}^{p} b_{D}^{(j)}(p\l^{(j)} T)\right)&\leq\sum_{j=1}^{p}(\l^{(j)}/\l)\tFD^{(j)}(b_{D}^{(j)}(p\l^{(j)} T))\\
&\leq\sum_{j=1}^{p}(\l^{(j)}/\l)(p\l^{(j)} T)^{-1}\\
&=(\l T)^{-1}.
\end{align*}
Thus, the left side of \eqref{eq:bbounds} follows.

On the other hand, since $\FD$ is right continuous, we have for each $j=1,\ldots,p$:
\begin{equation*}
(\l^{(j)}/\l)\tFD^{(j)}(b_{D}(\l T))\leq\sum_{k=1}^{p}(\l^{(k)}/\l)\tFD^{(k)}(b_{D}(\l T))\leq(\l T)^{-1},
\end{equation*}
whence
\begin{equation*}
\tFD^{(j)}(b_{D}(\l T))\leq(\l^{(j)} T)^{-1}.
\end{equation*}
Therefore, the right side of \eqref{eq:bbounds} follows.
\end{proof}

The distribution {$\FD=\sum_{j=1}^p (\lambda^{(j)}/\lambda)
F_D^{(j)}$ of session durations of superimposed streams}
 is a function of $T$ since {$\lambda^{(j)}$ and $\lambda $ depend
   on $T$.} Nevertheless, {$\bar F_D$  behaves as a regularly varying function.}

\begin{mylem}\label{lem:frv}
Under the assumption \eqref{eq:novanishprop}
\begin{equation}\label{eq:frv}
\lim_{T\to\infty}\frac{\tFD(Tu)}{\tFD(T)}=\lim_{T\to\infty}\l T\tFD(b_{D}(\l T) u)=u^{-\aD},\quad u>0,
\end{equation}
{and therefore, 
in $M_{+}(0,\infty]$,
\begin{equation}\label{eqn:sid4}
\frac{\tFD(Tdu)}{\tFD(T)}\xrightarrow{v}\aD u^{-(\aD+1)}du,\quad\quad T\to\infty.
\end{equation}
}
\end{mylem}
\begin{proof}
Note that $\tFD\in RV_{-\aD}$ for each fixed $T$. However, because $\FD$ varies with $T$, the limit is not straightforward.

Fix an arbitrary $u>0$. We start by writing
\begin{align*}
\frac{\tFD(Tu)}{\tFD(T)}&=\frac{\sum_{j:\aD^{(j)}=\aD}(\l^{(j)}/\l)\tFD^{(j)}(Tu)}{\tFD(T)}+\sum_{j:\aD^{(j)}>\aD}(\l^{(j)}/\l)\frac{\tFD^{(j)}(Tu)}{\tFD(T)}\\
&=: B+\sum_{j:\aD^{(j)}>\aD}(\l^{(j)}/\l)C_{j},
\end{align*}
and additionally write
\begin{align*}
B^{-1}&=\frac{\sum_{j:\aD^{(j)}=\aD}(\l^{(j)}/\l)\tFD^{(j)}(T)}{\sum_{j:\aD^{(j)}=\aD}(\l^{(j)}/\l)\tFD^{(j)}(Tu)}+\sum_{j:\aD^{(j)}>\aD}(\l^{(j)}/\l)\frac{\tFD^{(j)}(T)}{\sum_{k:\a_{k}=\aD}(\l_{k}/\l)\tFD^{(k)}(Tu)}\\
&=:B_{1}+\sum_{j:\aD^{(j)}>\aD}(\l^{(j)}/\l)B_{2,j}.
\end{align*}

Thus, the first limit in \eqref{eq:frv} will follow by proving
$B_{1}\to u^{\aD}$, $B_{2,j}\to0$ and $C_{j}\to0$ as $T\to\infty$, for
all $j$ such that $\aD^{(j)}>\a_{{D}}$. 

First, by Potter's bounds applied to the regular
variation of each $\tFD^{(j)}$, {we have} as $T\to\infty$,
$$
B_{1} \sim\frac{\sum_{j:\aD^{(j)}=\aD}(\l^{(j)}/\l)\tFD^{(j)}(T)}{\sum_{j:\aD^{(j)}=\aD}(\l^{(j)}/\l)\tFD^{(j)}(T)u^{-\aD}}
=u^{\aD}.
$$

Now, consider $B_{2,j}$ for $\aD^{(j)}>\aD$. Choose an arbitrarily large $z>\min\{u^{-\aD^{(j)}},u^{\aD^{(j)}}\}$. By regular variation, for $T$ sufficiently large:
\begin{equation*}
\frac{\tFD^{(k)}(Tu)}{\tFD^{(j)}(Tu)}>z,\quad\quad \frac{\tFD^{(k)}(T)}{\tFD^{(j)}(T)}>z,
\end{equation*}
for all $k$ such that $\aD^{(k)}=\aD$. In addition
\begin{equation*}
\frac{\tFD^{(j)}(Tu)}{\tFD^{(j)}(T)} > u^{-a_{D}^{(j)}} - z^{-1},\quad\quad\frac{\tFD^{(j)}(T)}{\tFD^{(j)}(Tu)} > u^{a_{D}^{(j)}} - z^{-1}.
\end{equation*}
Furthermore, the assumption \eqref{eq:novanishprop} means there exists
$d>0$ such that for all $T$ sufficiently large, there is some
$k':=k'(T)$ such that $\aD^{(k')}=\aD$ and $\l^{(k')}/\l>d$. Hence for
$T$ sufficiently large: 
\begin{align*}
B_{2,j}^{-1}&=\sum_{k:{\a_{D}^{(k)}=\a_D}}({\l^{(k)}}/\l)\frac{\tFD^{(k)}(Tu)}{\tFD^{(j)}(T)}
=\sum_{k:{\a_{D}^{(k)}=\a_D}}({\l^{(k)}}/\l)\frac{\tFD^{(k)}(Tu)}{\tFD^{(j)}(Tu)}\frac{\tFD^{(j)}(Tu)}{\tFD^{(j)}(T)}
> d z (u^{-{\aD}}-z^{-1}).
\end{align*}
This shows that $B_{2,j}^{-1}$ can be made arbitrarily large for $T$ sufficiently large, whence $B_{2,j}\to0$ as $T\to\infty$.

Similarly, consider $C_{j}$, and
\begin{align*}
C_{j}^{-1}&\geq \sum_{k:\aD^{(k)}=\aD}(\l^{(k)}/\l)\frac{\tFD^{(k)}(T)}{\tFD^{(j)}(Tu)}
=
\sum_{k:\aD^{(k)}=\aD}(\l^{(k)}/\l)\frac{\tFD^{(k)}(T)}{\tFD^{(j)}(T)}\frac{\tFD^{(j)}(T)}{\tFD^{(j)}(Tu)}
>d z(u^{{\aD}}-z^{-1}).
\end{align*}
This shows that $C_{j}^{-1}$ can be made arbitrarily large for $T$ sufficiently large, which completes the first part of the Lemma.

For the second limit in \eqref{eq:frv}, recall that $z<b_{D}(\l T)$ iff $1/\tFD(z)<\l T$ for each $T$. For $\e>0$, setting $z=b_{D}(\l T)(1-\e)$ and $z=b_{D}(\l T)(1+\e)$ yields 
\begin{equation*}
\frac{\tFD(b_{D}(\l T)(1+\e))}{\tFD(b_{D}(\l T))}\leq\frac{1}{\l T\tFD(b_{D}(\l T) )}\leq\frac{\tFD(b_{D}(\l T)(1-\e))}{\tFD(b_{D}(\l T))}.
\end{equation*}
Letting $T\to\infty$ and using Lemma \ref{lem:bbounds} and the first limit gives
\begin{equation*}
(1+\e)^{-\aD}\leq\frac{1}{\l T\tFD(b_{D}(\l T) )}\leq(1-\e)^{-\aD}.
\end{equation*}
Because $\e$ is arbitrary, then
\begin{equation*}
\lim_{T\to\infty}\l T\tFD(b_{D}(\l T) )=1.
\end{equation*}
Therefore
\begin{equation*}
\lim_{T\to\infty}\l T\tFD(b_{D}(\l T) u)=\lim_{T\to\infty}\l T\tFD(b_{D}(\l T))\lim_{T\to\infty}\frac{\tFD(b_{D}(\l T)u)}{\tFD(b_{D}(\l T))}=u^{-\aD}.
\end{equation*}
{The final statement about vague convergence
follows the proof} of \citet[][Theorem 3.6]{resnickbook:2007}.
\end{proof}

Even though $\FD$ depends on $T$, {a version} of Potter's bounds holds.

\begin{mylem}\label{lem:fpotter}
Let $\d>0$. Under the assumption \eqref{eq:novanishprop}, there exists $T_{0}=T_{0}(\d)>0$ such that for all $T\geq T_{0}$, $Tu\geq T_{0}$:
\begin{equation*}
\frac{\tFD(Tu)}{\tFD(T)}\leq(1+\d)u^{-\aD}\max\{u^{-\d},u^{\d}\}.
\end{equation*}
\end{mylem}
\begin{proof}
Observe
\begin{equation*}
\frac{\tFD(Tu)}{\tFD(T)}=\frac{\sum_{j=1}^{p}(\l^{(j)}/\l)\tFD^{(j)}(T)\frac{\tFD^{(j)}(Tu)}{\tFD^{(j)}(T)}}{\sum_{j=1}^{p}(\l^{(j)}/\l)\tFD^{(j)}(T)}\leq \bigvee_{j=1}^{p}\frac{\tFD^{(j)}(Tu)}{\tFD^{(j)}(T)}.
\end{equation*}

By Potter bounds \citep[See e.g.][Theorem 1.5.6]{bingham:goldie:teugels:1987}, for all $j=1,\ldots,p$ there exists $T_{j}=T_{j}(\d)$ such that
\begin{equation*}
\frac{\tFD^{(j)}(Tu)}{\tFD^{(j)}(T)}\leq (1+\d)u^{-\aD^{(j)}}\max\{u^{-\d},u^{\d}\}\leq (1+\d)u^{-\aD}\max\{u^{-\d},u^{\d}\},
\end{equation*}
for $T\geq T_{j}$, $Tu\geq T_{j}$. Therefore, the result holds for $T_{0}=\bigvee_{j=1}^{p}T_{j}$.
\end{proof}

We now study $L_{t}(s,u)$, as defined in \eqref{eq:length}.

\begin{mylem}\label{lem:lengthprop}
The length of the subinterval of $[0,t]$ during which the session $(s,u,r)$ transmits data, namely $L_{t}(s,u)$ in \eqref{eq:length}, satisfies the following properties:
\begin{enumerate}[(i)]
\item \textbf{Scaling property:} For $C>0$,
\begin{equation*}
CL_{t}(s,u)=L_{Ct}(Cs,Cu).
\end{equation*}
\item \textbf{Bounds:}
\begin{equation}
L_{t}(s,u)\leq t\wedge u. \label{eqn:sid4}
\end{equation}
\item \textbf{Integrals:} For $1<\gamma<2$ and nonnegative $t_{1},t_{2}$,
\begin{equation*}
\int_{-\infty}^{\infty}L_{t_{1}}(s,u)ds=ut_{1},
\end{equation*}
and
\begin{align*}
\int_{-\infty}^{\infty}\int_{0}^{\infty}&L_{t_{1}}(s-u,u)1_{[0,t_{2}]}(s)u^{-\gamma}duds\\
&=\frac{1}{(\gamma-1)(2-\gamma)(3-\gamma)}\bigg\{(t_{2}^{3-\gamma}-(t_{2}-t_{1})^{3-\gamma})1_{t_{1}<t_{2}}+t_{2}^{3-\gamma}1_{t_{1}\geq t_{2}}\bigg\}.
\end{align*}
\end{enumerate}
\end{mylem}
\begin{proof}
The scaling property and the bounds follow directly from \eqref{eq:length}.

Now, the first part of Property (iii) is readily checked by using the first integral in \eqref{eq:length} after reversing the order of integration. Finally, the second part of Property (iii) can be derived by writing
\begin{equation*}
L_{t_{1}}(s-u,u)=\begin{cases}
0,&\textrm{$s<0$ or $s>u+t_{1}$},\\
s,&0\leq s\leq u\wedge t_{1},\\
t_{1},&t_{1}\leq s\leq u,\\
u,&u\leq s\leq t_{1},\\
t_{1}-s+u,&u\vee t_{1}\leq s\leq u+t_{1},
\end{cases}
\end{equation*}
and integrating accordingly. Observe that the four regions in which $L_{t_{1}}(s-u,u)$ is nonzero correspond to those of the basic decomposition in \citet[][Eq. 4.1]{mikosch:resnick:rootzen:stegeman:2002}.

\end{proof}

The next lemma helps obtain approximations to the cumulative input of the
aggregated  streams.

\begin{mylem}\label{lem:chv}
For any $a,T>0$, we have
\begin{align*}
\frac 1a (A(Tt)-\l\mD \mR Tt)&=\frac 1a\int_{-\infty}^{\infty}\int_{0}^{\infty}\int_{0}^{\infty}rL_{Tt}(Ts,u)\cprm{N}(Tds,du,dr)\\
&=\frac {T}{a}\int_{-\infty}^{\infty}\int_{0}^{\infty}\int_{0}^{\infty}rL_{t}(s,u)\cprm{N}(Tds,Tdu,dr)\\
&=\int_{-\infty}^{\infty}\int_{0}^{\infty}\int_{0}^{\infty}rL_{t}(s,u)\cprm{N}(Tds,Tdu, (a/T) dr).
\end{align*}
\end{mylem}

\begin{proof}
All the relations here follow from several ways to change variables in \eqref{eq:caginp} and using the scaling property of $L_{t}(s,u)$. See Lemma \ref{lem:lengthprop}.
\end{proof}

Our limit theorems are proved by verifying convergence of finite
dimensional distributions {for various} processes. The following
is required.

\begin{myprop}\label{prop:chf}
For arbitrary $m\geq1$, $0\leq t_{1},\ldots,t_{m}$, and real $z_{1},\ldots,z_{m}$, define
\begin{equation}\label{eq:gfun}
g(s,u,r)=\exp\left\{i\sum_{j=1}^{m}z_{j}rL_{t_{j}}(s,u)\right\}-1-i\sum_{j=1}^{m}z_{j}rL_{t_{j}}(s,u),
\end{equation}
and
\begin{equation}\label{eq:hfun}
h(s,u,r)=i\left(\exp\left\{i\sum_{j=1}^{m}z_{j}ur1_{(0,t_{j})}(s)\right\}-1\right)\sum_{k=1}^{m}z_{k}1_{[0,t_{k}]}(s)ru^{-\aD}.
\end{equation}

\begin{enumerate}[(a)]

\item For any $a,T>0$, the characteristic function of the finite-dimensional distributions \emph{(fidi chf)} of the process $\{(1/a)(A(Tt)-\l\mD \mR Tt);t\geq0\}$ is given by
\begin{align}
\ln E \exp&\left\{i\sum_{j=1}^{m}z_{j}\left[\frac 1a (A(Tt_{j})-\l\mD \mR Tt_{j})\right]\right\}=\int_{-\infty}^{\infty}\int_{0}^{\infty}\int_{0}^{\infty}g(s,u,r)EN(Tds,Tdu,(a/T)dr)\label{eq:csaginpchf1}\\
&=\int_{-\infty}^{\infty}\int_{0}^{\infty}\int_{0}^{\infty}g_{u}(s-u,u,r)\l
T\tFD(Tu)dsdu\FR((a/T)dr).\label{eq:csaginpchf2} 
\end{align}
where $g_{u}$ is the partial derivative of $g$ with respect to $u$.

\item The \fidichf of the limit processes in Corollary \ref{cor:constprop} are given as follows.
\begin{enumerate}[(i)]
\item The \fidichf of the limit process under Scenario $\Fast$ and $E[( R_{1}^{(1)} )^{2}]<\infty$ is given by
\begin{align}\label{eq:limfichf}
\ln E\exp&\left\{i\sum_{j=1}^{m}z_{j}E[( R_{1}^{(1)} )^{2}]^{1/2}\sigma_{B_{H}(1)} B_{H}(t_{j})\right\}\notag\\
&=-\frac 12 E[( R_{1}^{(1)} )^{2}]\sum_{j=1}^{m}\sum_{k=1}^{m}z_{i}z_{j}\sigma_{B_{H}(1)}^{2}\frac{1}{2}\left(t_{i}^{2H}+t_{j}^{2H}-|t_{i}-t_{j}|^{2}\right),
\end{align}
where $B_{H}$ is fractional Brownian motion with
\begin{equation}\label{eq:varbh1}
\sigma_{B_{H}(1)}^{2}=\frac{2}{(\aD-1)(2-\aD)(3-\aD)},
\end{equation}
and $H=(3-\aD)/2$.

\item The \fidichf of the limit process under Scenario $\Fast$ and $\tFR\in RV_{-\aR}, 1<\aR<2,$ is given by
\begin{equation}\label{eq:limfiichf}
\ln E\exp\left\{i\sum_{j=1}^{m}z_{j}Z_{\aD,\aR}(t_{j})\right\}=\int_{-\infty}^{\infty}\int_{0}^{\infty}\int_{0}^{\infty}g(s,u,r)EN^{\infty}_{\aD,\aR}(ds,du,dr).
\end{equation}

\item The \fidichf of the limit process under Scenario $\Moderate$ is given by
\begin{equation}\label{eq:limmchf}
\ln E\exp\left\{i\sum_{j=1}^{m}z_{j}cY_{\aD}(t_{j}/c)\right\}=c^{\aD-1}\int_{-\infty}^{\infty}\int_{0}^{\infty}\int_{0}^{\infty}g(s,u,r)EN^{\infty}_{\aD,\FR}(ds,du,dr).
\end{equation}

\item Finally, the \fidichf of the limit process under Scenario $\Slow$ is given by
\begin{equation}\label{eq:limschf}
\ln E\exp\left\{i\sum_{j=1}^{m}z_{j}E[( R_{1}^{(1)} )^{\aD}]^{1/\aD}\Lambda_{a_{D}}(t_{j})\right\}=\int_{-\infty}^{\infty}\int_{0}^{\infty}\int_{0}^{\infty}h(s,u,r)dsdu\FR(dr).
\end{equation}

\end{enumerate}
\end{enumerate}
\end{myprop}

\begin{proof}

Given {\eqref{eq:csaginpchf1}},  \eqref{eq:csaginpchf2} is readily derived using
 integration by parts and the change of variables $s\mapsto s+u$. Moreover, \eqref{eq:limfichf} follows from the fact that $B_{H}$ is fractional Brownian motion.  The remaining parts are a consequence of the following property of Poisson random measures \citep[See e.g.][]{rosinski:rajput:1989}:
\begin{equation*}
\ln E\exp\left\{i\int f(x)\cprm{\xi}(dx)\right\}=\int \left(e^{if(x)}-1-if(x)\right)E\xi(dx),
\end{equation*}
if
\begin{equation*}
\int \left(f^{2}(x)\wedge |f(x)|\right)E\xi(dx)<\infty.
\end{equation*}

For now, let us focus on \eqref{eq:csaginpchf1}, \eqref{eq:limfiichf} and \eqref{eq:limmchf}. By Lemma \ref{lem:chv}, the exponent in the left side of \eqref{eq:csaginpchf1} and \eqref{eq:limfiichf} is of the form
\begin{equation}\label{eq:exponentchv}
i\int_{-\infty}^{\infty}\int_{0}^{\infty}\int_{0}^{\infty}\sum_{j=1}^{m}z_{j}rL_{t_{j}}(s,u)\cprm{\xi}(ds,du,dr),
\end{equation}
for a $\PRM$ $\xi$, while the exponent in the left side of \eqref{eq:limmchf} is $c^{\aD-1}$ times \eqref{eq:exponentchv}, using Lemma \ref{lem:lengthprop} and the change of variables $s\mapsto s/c,u\mapsto u/c$. Thus, it suffices to check that
\begin{equation}\label{eq:chfcondition}
\int_{-\infty}^{\infty}\int_{0}^{\infty}\int_{0}^{\infty}\left(\sum_{j=1}^{m}z_{j}rL_{t_{j}}(s,u)\right)^{2}\bigwedge\left|\sum_{k=1}^{m}z_{k}rL_{t_{k}}(s,u)\right|E\xi(ds,du,dr)<\infty.
\end{equation}

Bounds and integral results for $L_{t}(s,u)$ in Lemma \ref{lem:lengthprop} (ii) and (iii) needed.

First observe that
\begin{equation*}
\int_{-\infty}^{\infty}\int_{0}^{\infty}\int_{0}^{\infty}\left|\sum_{j=1}^{m}z_{j}rL_{t_{j}}(s,u)\right|EN(Tds,Tdu,(a/T)dr)\leq\frac Ta\sum_{j=1}^{m}|z_{j}|\l \mD \mR t_{j},
\end{equation*}
which proves \eqref{eq:csaginpchf1}.

In order to prove \eqref{eq:limfiichf}, split the corresponding integral \eqref{eq:chfcondition} into two parts $I_{(<)}$ and $I_{(>)}$, according to the two domains of integration $D_{(<)}=\{ur<1\}$ and $D_{(>)}=\{ur>1\}$. This yields
\begin{align*}
I_{(<)}&\leq\sum_{j=1}^{m}\sum_{k=1}^{m}|z_{j}z_{k}|\int_{-\infty}^{\infty}\int_{0}^{\infty}\int_{0}^{1/u}r^{2}L_{t_{j}}(s,u)L_{t_{k}}(s,u)EN^{\infty}_{\aD,\aR}(ds,du,dr)\\
&\leq\sum_{j=1}^{m}\sum_{k=1}^{m}\frac{|z_{j}z_{k}|t_{j}\aD\aR}{2-\aR}\left(\frac{1}{\aR-\aD}+\frac{t_{k}}{1-\aR+\aD}\right),
\end{align*}
and
\begin{align*}
I_{(>)}&\leq\sum_{j=1}^{m}|z_{j}|\int_{-\infty}^{\infty}\int_{0}^{\infty}\int_{1\vee u^{-1}}^{\infty}rL_{t_{j}}(s,u)EN^{\infty}_{\aD,\aR}(ds,du,dr)\\ 
&+\sum_{j=1}^{m}\sum_{k=1}^{m}|z_{j}z_{k}|\int_{-\infty}^{\infty}\int_{0}^{\infty}\int_{u^{-1}}^{1\vee u^{-1}}r^{2}L_{t_{j}}(s,u)L_{t_{k}}(s,u)EN^{\infty}_{\aD,\aR}(ds,du,dr)\\
&\leq\sum_{j=1}^{m}\frac{|z_{j}|t_{j}\aD\aR}{\aR-1}\left(\frac{1}{\aR-\aD}+\frac{1}{\aD}\right)+\sum_{j=1}^{m}\sum_{k=1}^{m}\frac{|z_{j}z_{k}t_{j}t_{k}|}{(\aD-1)(2-\aR)},
\end{align*}
whence \eqref{eq:limfiichf} holds.

Similarily, split the integral \eqref{eq:chfcondition} corresponding to the process \eqref{eq:limmchf} into two parts $J_{(<)}$ and $J_{(>)}$, according to the two domains of integration $D_{(<)}$ and $D_{(>)}$, which yields
\begin{align*}
J_{(<)}&\leq\sum_{j=1}^{m}\sum_{k=1}^{m}|z_{j}z_{k}|\int_{-\infty}^{\infty}\int_{0}^{\infty}\int_{0}^{r^{-1}}r^{2}L_{t_{j}}(s,u)L_{t_{k}}(s,u)EN^{\infty}_{\aD,\FR}(ds,du,dr)\\
&\leq\sum_{j=1}^{m}\sum_{k=1}^{m}\frac{|z_{j}z_{k}|t_{j}\aD}{2-\aD}E[( R_{1}^{(1)} )^{\aD}],
\end{align*}
and
\begin{align*}
J_{(>)}&\leq\sum_{j=1}^{m}|z_{j}|\int_{-\infty}^{\infty}\int_{0}^{\infty}\int_{r^{-1}}^{\infty}rL_{t_{j}}(s,u)EN^{\infty}_{\aD,\FR}(ds,du,dr)\\
&\leq\sum_{j=1}^{m}\frac{|z_{j}|t_{j}\aD}{\aD-1}E[( R_{1}^{(1)} )^{\aD}].
\end{align*}
This proves \eqref{eq:limmchf}.

Finally, the exponent in the left side of \eqref{eq:limschf} is
\begin{equation*}
i\int_{-\infty}^{\infty}\int_{0}^{\infty}\int_{0}^{\infty}\sum_{j=1}^{m}z_{j}ur1_{[0,t_{j}]}(s)\cprm{N^{\infty}_{\aD,\FR}}(ds,du,dr).
\end{equation*}
Analogous to the proof of \eqref{eq:limmchf}, it is readily shown that
\begin{equation*}
\int_{-\infty}^{\infty}\int_{0}^{\infty}\int_{0}^{\infty}\left(\sum_{j=1}^{m}z_{j}ur1_{[0,t_{j}]}(s)\right)^{2}\bigwedge\left|\sum_{k=1}^{m}z_{k}ur1_{[0,t_{j}]}(s)\right|EN^{\infty}_{\aD,\FR}(ds,du,dr)<\infty,
\end{equation*}
whence the left side of \eqref{eq:limschf} is equal to
\begin{equation*}
\int_{-\infty}^{\infty}\int_{0}^{\infty}\int_{0}^{\infty}\left(\exp\bigg\{i\sum_{j=1}^{m}z_{j}ur1_{[0,t_{j}]}(s)\bigg\}-1-\sum_{j=1}^{k}z_{j}ur1_{[0,t_{j}]}(s)\right)EN^{\infty}_{\aD,\FR}(ds,du,dr).
\end{equation*}
The result now follows after an integration by parts in the variable $u$.
\end{proof}

Finally, the following result is used to get upper bounds for some integrands throughout the proof of Theorem \ref{thm:novanishprop}.

\begin{mylem}\label{lem:ebounds}
For $0\leq\z\leq1$ and $x\in\R$:
\begin{equation}
|e^{ix}-1|\leq2^{1-\z}|x|^{\z},\label{eqn:sid1}
\end{equation}
\begin{equation}
|e^{ix}-1-ix|\leq d_{\z}|x|^{\z+1},\label{eqn:sid2}
\end{equation}
where $d_{\z}>0$, and for real numbers $x_{1},\ldots,x_{m}$:
\begin{equation}
\bigg(\sum_{j=1}^{m}|x_{j}|\bigg)^{\z}\leq\sum_{j=1}^{m}|x_{j}|^{\z}.\label{eqn:sid3}
\end{equation}
\end{mylem}
\begin{proof}
Without loss of generality, fix $x\not=0$. Define $f:[0,1]\to\R$, $f(\z)=(1-\z)\ln 2 +\z\ln |x|$. We can readily check that
\begin{equation*}
\ln|e^{ix}-1|\leq f(\z),\quad \z=0,1,
\end{equation*}
by taking logarithms in both sides of
$|e^{ix}-1|\leq2\wedge|x|$. Since $f(\z)$ is linear in $\z$, $f(\z)$
is either nondecreasing or nonincreasing on $[0,1]$. Hence,
{\eqref{eqn:sid1}}
 holds. Using a similar strategy, we can prove {\eqref{eqn:sid2}}.

For {\eqref{eqn:sid3}},
 assume without loss of generality that $0<|x_{1}|\leq|x_{2}|$, thus $0<|x_{1}/x_{2}|\leq 1$. By Bernoulli's inequality \citep[see e.g.][p. 36]{mitrinovic:vasic:1970}:
\begin{equation*}
(1+|x_{1}/x_{2}|)^{\z}\leq1+\z |x_{1}/x_{2}|\leq 1+ |x_{1}/x_{2}|^{\z}.
\end{equation*}
Multiplying both sides by $|x_{2}|^{\z}$ proves 
{\eqref{eqn:sid3}}
for $m=2$ {and the proof for general $m$ follows by induction. }
\end{proof}

\section{Acknowledgements}

S. I. Resnick was partially supported by ARO Contract W911NF-10-1-0289
at Cornell University.

Dan Eckstrom and Ed Kiefer (CIT-Network Communication Services, Cornell University) and Eric Johnson (ORIE, Cornell) were very helpful with arrangements and collection of the Cornell flow data.

\bibliographystyle{spbasic}
\bibliography{luisloref}

\begin{thebibliography}{41}
\providecommand{\natexlab}[1]{#1}
\providecommand{\url}[1]{{#1}}
\providecommand{\urlprefix}{URL }
\expandafter\ifx\csname urlstyle\endcsname\relax
  \providecommand{\doi}[1]{DOI~\discretionary{}{}{}#1}\else
  \providecommand{\doi}{DOI~\discretionary{}{}{}\begingroup
  \urlstyle{rm}\Url}\fi
\providecommand{\eprint}[2][]{\url{#2}}

\bibitem[{Arlittson and Williamson(1996)}]{arlittson:williamson:1996}
Arlittson MF, Williamson CL (1996) Web server workload characterization: The
  search for invariants (extended version). In: Proceedings of the 1996 ACM
  SIGMETRICS Conference, ACM, New York, Philadelphia, PA, pp 126--137

\bibitem[{Billingsley(1999)}]{billingsley:1999}
Billingsley P (1999) Convergence of probability measures, 2nd edn. Wiley Series
  in Probability and Statistics, John Wiley \& Sons, Inc.

\bibitem[{Bingham et~al(1987)Bingham, Goldie, and
  Teugels}]{bingham:goldie:teugels:1987}
Bingham NH, Goldie CM, Teugels JL (1987) Regular variation, Encyclopedia of
  Mathematics and its Applications, vol~27. Cambridge University Press

\bibitem[{Brockwell and Davis(1991)}]{brockwell:davis:1991}
Brockwell PJ, Davis RA (1991) Time Series: Theory and Methods, 2nd edn.
  Springer-Verlag, New York

\bibitem[{Cisco Systems, Inc.(2007)}]{cisco:2007}
Cisco Systems, Inc. (2007) Introduction to Cisco IOS NetFlow - A Technical
  Overview. Cisco Systems, Inc., San Jose, CA. USA

\bibitem[{Coles(2001)}]{coles:2001}
Coles S (2001) An Introduction to Statistical Modeling of Extreme Values.
  Springer Series in Statistics, Springer

\bibitem[{Crovella and Bestavros(1997)}]{crovella:bestavros:1997}
Crovella ME, Bestavros A (1997) Self-similarity in world wide web traffic:
  Evidence and possible causes. IEEM/ACM Transactions on Networking
  5(6):845--846

\bibitem[{Cunha et~al(1995)Cunha, Bestavros, and
  Crovella}]{cunha:bestavros:crovella:1995}
Cunha CR, Bestavros A, Crovella ME (1995) Characteristics of www client-based
  traces. Technical report BU-CS-95-010, Computer Science Department, Boston
  University

\bibitem[{Guerin et~al(2003)Guerin, Nyberg, Perrin, Resnick, Rootz\'en, and
  St\u{a}ric\u{a}}]{guerin:nyberg:perrin:resnick:rootzen:starica:2003}
Guerin C, Nyberg H, Perrin O, Resnick SI, Rootz\'en H, St\u{a}ric\u{a} C (2003)
  Empirical testing of the infinite source poisson data traffic model.
  Stochastic Models 19(2):151--200

\bibitem[{de~Haan and Ferreira(2006)}]{dehaan:ferreira:2006}
de~Haan L, Ferreira A (2006) Extreme Value Theory: An Introduction.
  Springer-Verlag, New York

\bibitem[{de~Haan and Resnick(1998)}]{dehaan:resnick:1998}
de~Haan L, Resnick SI (1998) On asymptotic normality of the hill estimator.
  Stochastic Models 14(4):849--866

\bibitem[{Hill(1975)}]{hill:1975}
Hill BM (1975) A simple general approach to inference about the tail of a
  distribution. The Annals of Statistics 3(5):1163--1174

\bibitem[{Jain and Dovrolis(2005)}]{jain:dovrolis:2005}
Jain M, Dovrolis C (2005) End-to-end estimation of the available bandwidth
  variation range. In: Proceedings of the 2005 ACM SIGMETRICS International
  Conference on Measurement and Modeling of computer systems, ACM, pp 265--276

\bibitem[{Jin et~al(2007)Jin, Bali, Duncan, and
  Frost}]{jin:bali:duncan:frost:2007}
Jin Y, Bali S, Duncan TE, Frost VS (2007) Predicting properties of congestion
  events for a queueing system with fbm traffic. IEEM/ACM Transactions on
  Networking 15(5):1098--1108

\bibitem[{Kaj and Taqqu(2008)}]{kaj:taqqu:2008}
Kaj I, Taqqu MS (2008) In and Out of Equilibrium 2, Progress in Probability,
  vol~60, Birkh{\"a}user Basel, chap Convergence to Fractional Brownian Motion
  and to the Telecom Process: the Integral Representation Approach, pp 383--427

\bibitem[{Kallenberg(1984)}]{kallenberg:1984}
Kallenberg O (1984) Random measures. Akademie-Verlag

\bibitem[{Kilpi and Norros(2002)}]{kilpi:norros:2002}
Kilpi J, Norros I (2002) Testing the gaussian approximation of aggregate
  traffic. In: Proceedings of the 2nd ACM SIGCOMM Workshop on Internet
  measurment, ACM, Marseilles, France, Session 2: modeling, pp 49--61

\bibitem[{Kingman(1993)}]{kingman:1993}
Kingman JFC (1993) Poisson Processes. Oxford Studies in Probability, Oxford
  University Press

\bibitem[{Kortebi et~al(2005)Kortebi, Muscariello, Oueslati, and
  Roberts}]{kortebi:muscariello:oueslati:roberts:2005}
Kortebi A, Muscariello L, Oueslati S, Roberts J (2005) Evaluating the number of
  active flows in a scheduler realizing fair statistical bandwidth sharing. In:
  SIGMETRICS '05: Proceedings of the 2000 ACM SIGMETRICS international
  conference on Measurement and modeling of computer systems, ACM, pp 217--228

\bibitem[{Kurtz(1996)}]{kurtz:1996}
Kurtz TG (1996) Limit theorems for workload input models. In: Kelly F, Zachary
  S, Ziedins I (eds) Stochastic Networks: Theory and Applications, no.~4 in
  Royal Statistical Society Lecture Note Series, Clarendon Press, Oxford, pp
  119--139

\bibitem[{Leland et~al(1994)Leland, Taqqu, Willinger, and
  Wilson}]{leland:taqqu:willinger:wilson:1994}
Leland WE, Taqqu MS, Willinger W, Wilson DV (1994) On the self-similar nature
  of ethernet traffic (extended version). IEEM/ACM Transactions on Networking
  2(1):1--15

\bibitem[{L\'opez-Oliveros and Resnick(2009)}]{lopez-oliveros:resnick:2009}
L\'opez-Oliveros L, Resnick SI (2009) Extremal dependence analysis of network
  sessions. Extremes, DOI: 101007/s10687-009-0096-4

\bibitem[{Maulik et~al(2002)Maulik, Resnick, and
  Rootz{\'e}n}]{maulik:resnick:rootzen:2002}
Maulik K, Resnick SI, Rootz{\'e}n H (2002) Asymptotic independence and a
  network traffic model. Journal of Applied Probability 39(4):671--699

\bibitem[{Mc~Neil et~al(2005)Mc~Neil, Frey, and
  Embrechts}]{mcneil:frey:embrechts:2005}
Mc~Neil AJ, Frey R, Embrechts P (2005) Quantitative Risk Management. Princeton
  Series in Finance, Princeton University Press, Princeton, NJ, concepts,
  Techniques and Tools

\bibitem[{van~de Meent and Mandjes(2005)}]{meent:mandjes:2005}
van~de Meent R, Mandjes M (2005) Evaluation of `user-oriented' and `black-box'
  traffic models for link provisioning. In: Proceedings of the 1st EuroNGI
  Conference on Next Generation Internet Networks Traffic Engineering, IEEE,
  Rome, Italy, pp 380--387

\bibitem[{van~de Meent et~al(2006)van~de Meent, Mandjes, and
  Pras}]{meent:mandjes:pras:2006}
van~de Meent R, Mandjes M, Pras A (2006) Gaussian traffic everywhere? In: ICC
  '06. IEEE International Conference on Communications, 2006., Istanbul,
  Turkey, vol~2, pp 573--578

\bibitem[{Mikosch et~al(2002)Mikosch, Resnick, Rootz\'en, and
  Stegeman}]{mikosch:resnick:rootzen:stegeman:2002}
Mikosch T, Resnick SI, Rootz\'en H, Stegeman A (2002) Is network traffic
  approximated by stable {L}{\'e}vy motion or fractional {B}rownian motion?
  Annals of Applied Probability 12(1):23--68

\bibitem[{Mitrinovi\'c and Vasi\'c(1970)}]{mitrinovic:vasic:1970}
Mitrinovi\'c DS, Vasi\'c PM (1970) Analytic Inequalities, Die Grundlehren der
  mathematischen Wissenschaften in Einzeldarstellungen mit besonderer
  Ber{\"u}cksichtigung der Anwendungsgebiete, vol 165. Springer-Verlag, Berlin,
  New York

\bibitem[{Resnick(1986)}]{resnick:1986}
Resnick SI (1986) Point processes, regular variation and weak convergence.
  Advances in Applied Probability 18(1):66--138

\bibitem[{Resnick(1987)}]{resnick:1987}
Resnick SI (1987) Extremes Values, Regular Variation and Point Processes.
  Springer-Verlag

\bibitem[{Resnick(2003)}]{resnick:2003}
Resnick SI (2003) Modeling Data Networks, in SemStat: Seminaire Europeen de
  Statistique, Extreme Values in Finance, Telecommunications, and the
  Environment, Chapman-Hall, London, pp 287--372

\bibitem[{Resnick(2007)}]{resnickbook:2007}
Resnick SI (2007) Heavy-Tail Phenomena: Probabilistic and Statistical Modeling.
  Springer Series in Operations Research and Financial Engineering,
  Springer-Verlag, New York

\bibitem[{Rosi\'nski and Rajput(1989)}]{rosinski:rajput:1989}
Rosi\'nski J, Rajput BS (1989) Spectral representations of infinitely divisible
  processes. Probability Theory and Related Fields 82(3):451--487

\bibitem[{Samorodnitsky and Taqqu(1994)}]{samorodnitsky:taqqu:1994}
Samorodnitsky G, Taqqu MS (1994) Stable non-Gaussian random processes:
  stochastic models with infinite variance. Stochastic Modeling, Chapman \&
  Hall

\bibitem[{Sarvotham et~al(2002)Sarvotham, Wang, Riedi, and
  Baraniuk}]{sarvotham:wang:riedi:baraniuk:2002}
Sarvotham S, Wang X, Riedi RH, Baraniuk RG (2002) Additive and multiplicative
  mixture trees for network traffic modeling. In: ICASSP 2002: Proceedings of
  the International Conference on Acoustics, Speech, and Signal Processing,
  IEEE, Signal processing society, Orlando, Florida, vol~IV, pp 4040--4043

\bibitem[{Sarvotham et~al(2005)Sarvotham, Riedi, and
  Baraniuk}]{sarvotham:riedi:baraniuk:2005}
Sarvotham S, Riedi R, Baraniuk R (2005) Network and user driven alpha-beta
  on-off source model for network traffic. Computer Networks 48(3):335--350

\bibitem[{Shakkottai et~al(2005)Shakkottai, Brownlee, and
  Claffy}]{shakkottai:brownlee:claffy:2005}
Shakkottai S, Brownlee N, Claffy K (2005) A study of burstiness in tcp flows.
  In: Dovrolis C (ed) Proceedings of the 6th International Workshop in Passive
  and Active Network Measurement, PAM 2005, Springer, Boston, MA, Lecture Notes
  in Computer Science, vol 3431, pp 13--26

\bibitem[{Taqqu et~al(1997)Taqqu, Willinger, and
  Sherman}]{taqqu:willinger:sherman:1997}
Taqqu MS, Willinger W, Sherman R (1997) Proof of a fundamental result in
  self-similar traffic modeling. ACM SIGCOMM Computer Communication Review
  27(2):5--23

\bibitem[{Willinger et~al(1995)Willinger, Taqqu, Leland, and
  Wilson}]{willinger:taqqu:leland:wilson:1995}
Willinger W, Taqqu MS, Leland WE, Wilson DV (1995) Self-similarity in
  high-speed packet traffic: Analysis and modeling of ethernet traffic
  measurements. Statistical Science 10(1):67--85

\bibitem[{Willinger et~al(1997)Willinger, Taqqu, Sherman, and
  Wilson}]{willinger:taqqu:sherman:wilson:1997}
Willinger W, Taqqu MS, Sherman R, Wilson DV (1997) Self-similarity through high
  variability: Statistical analysis of ethernet lan traffic at the source
  level. IEEM/ACM Transactions on Networking 5(1):71--86

\bibitem[{Willinger et~al(1998)Willinger, Paxson, and
  Taqqu}]{willinger:paxson:taqqu:1998a}
Willinger W, Paxson V, Taqqu MS (1998) Self similarity and heavy tails:
  Structural modeling of network traffic, in A Practical Guide to Heavy Tails.
  Statistical Techniques and Applications, Birkh{\"a}user Boston Inc., Boston,
  MA, pp 27--53

\end{thebibliography}
\end{document}